\documentclass[12pt]{article}

  \usepackage{amsfonts,amssymb,amsthm,amsmath,cite,bbm}

\newtheorem{theo}{Theorem}
\newtheorem{lemma}[theo]{Lemma}
\newtheorem{coro}[theo]{Corollary}
\newtheorem{prop}[theo]{Proposition}
\def\Z{\mathbb Z}\def\Q{\mathbb Q}\def\N{\mathbb N}\def\R{\mathbb R}
\def\cen{\operatorname{cen}} %centroid
\def\dim{\operatorname{dim}} %dimension
\def\card{\operatorname{card}}%cardinality
\def\Int{\operatorname{int}} %interior
\def\aff{\operatorname{aff}} %affine hull
\def\Span{\operatorname{span}} %linear hull
\def\lcm{\operatorname{LCM}} %least common multiple

\def\op{\operatorname{\Pi}} %projection body operator
\def\oz{\operatorname{z}}
\def\ow{\operatorname{w}}
\def\oZ{\operatorname{Z}}
\def\oZp{\operatorname{Z^\prime}}

\def\cF{{\cal F}} \def\cG{{\cal G}}  \def\cS{{\cal S}}  \def\cL{{\cal L}}
\def\lpn{{\cal P}(\Z^n)} %lattice polytopes in n dimensions
\def\clpn{[{\cal P}](\Z^n)} %group of characteristc functions of lattice polytopes in n dimensions
\def\lpt{{\cal P}(\Z^2)}%lattice polygons
\def\lpm{{\cal P}(\Z^{n-1})} % lower dimensional lattice polytopes 
\def\lpth{{\cal P}(\Z^3)} %three dimensional lattice polytopes
\def\ckn{{\cal K}(\R^n)} %convex bodies in n dimensions
\def\ckt{{\cal K}(\R^2)} %convex bodies in 2 dimensions
\def\ckm{{\cal K}(\R^{n-1})} %convex bodies in n-1 dimensions

\def\cpn{{\cal P}(\R^n)} %convex polytopes in n dimensions

\def\cyl{\widetilde{T}_{n-1}} %n-dim prism
\def\cylt{\widetilde{T}_2} %3-dim prism

\newcommand{\dd}{\; \mathrm{d}}

\def\group{{\mathbb G}}

\def\sn{{{\mathbb S}^{n-1}}} %sphere

\def\dst{\ell_1} %discrete moment vector

\def\gln{\operatorname{GL}_n(\R)}
\def\slnz{\operatorname{SL}_n(\Z)} %Special linear group over \Z
\def\slmz{\operatorname{SL}_{n-1}(\Z)}
\def\sltz{\operatorname{SL}_2(\Z)}
\def\slt{\operatorname{SL}_2(\R)}

\def\sln{\operatorname{SL}_n(\R)}

\def\sym{\operatorname{Sym}(T_*)} %symmetry group of the regular simplex

\begin{document}
\author{K\'aroly J. B\"or\"oczky and Monika Ludwig}

\title{Minkowski valuations on lattice polytopes}

\date{ }

\maketitle

\begin{abstract}
A complete classification  is established of Minkowski valuations on lattice polytopes  that intertwine the special linear group over the integers and are translation invariant. In the contravariant case, the only such valuations are multiples of projection bodies. In the equi\-variant case, the only such valuations are generalized difference bodies combined with multiples of the newly defined discrete Steiner point.

\bigskip

\end{abstract}

\section{Introduction and statement of results}

Two classification theorems were critical in the beginning of the theory of valuations on convex sets: first, the Hadwiger theorem \cite{Hadwiger:V} for valuations on convex bodies (that is, compact convex sets) in $\R^n$ and second, the Betke \& Kneser theorem \cite{Betke:Kneser} for valuations on lattice polytopes (that is, convex polytopes with vertices in $\Z^n$). In recent years, numerous  classification results were established for valuations defined on convex bodies (see, for example, \cite{Alesker01,Alesker99,   Bernig:Fu, Haberl:Parapatits_centro, Haberl:Parapatits_crelle, Ludwig:matrix, Ludwig:Reitzner, Ludwig:Reitzner2, Parapatits:Wannerer, Wannerer14} and  \cite{Hadwiger:V,Klain:Rota,McMullen93,McMullen:Schneider}  for more  information). In particular, such results were obtained for convex-body valued valuations (see, for example, \cite{AleskerFaifman, LiYuanLeng, Ludwig:projection, Ludwig:convex, Haberl_sln, Abardia:Bernig, Abardia, Abardia15, Schuster:Wannerer, Schuster09, Wannerer2011, Haberl:blaschke, Parapatits:co, Parapatits:contravariant}). The aim of this article is to establish classification results for convex-body valued valuations defined on lattice polytopes. The question leads us to define and classify the discrete Steiner point.

A function $\oz$ defined on a family ${\cal F}$ of subsets of $\R^n$ with values in an abelian group (or more generally, an abelian monoid) 
is a valuation if 
\begin{equation}\label{valdef}
\oz(P)+\oz(Q)=\oz(P\cup Q)+\oz (P\cap Q)
\end{equation}
whenever $P,Q,P\cup Q,P\cap Q\in{\cal F}$ and $\oz(\emptyset)=0$.

In the Hadwiger theorem, ${\cal F}$ is the family, $\ckn$, of convex bodies and important further results regard the family, $\cpn$, of convex polytopes in $\R^n$. In both cases, the spaces are equipped with the topology coming from the Hausdorff metric.  A functional $\oz: \ckn \to \R$ is rigid motion invariant, if it is translation invariant 
and invariant with respect to orthogonal linear transformations.

\begin{theo}[Hadwiger \cite{Hadwiger:V}]\label{hugo}
A functional $\oz:\ckn\to \R$ is  a continuous and rigid motion invariant valuation if and only if  there exist constants $\,c_0$, $c_1,\dots, c_n\in\R$ such that 
$$\oz(K) =  c_0 \,V_0(K)+\dots+c_n\,V_n(K)$$
for every $K\in\ckn$. 
\end{theo}

\noindent
Here $V_0(K),\dots,V_n(K)$ are the intrinsic volumes of $K\in\ckn$. An elegant proof of this result is due to Klain \cite{Klain95} (or see \cite{Klain:Rota, Schneider:CB2}).

In the Betke \& Kneser theorem (and in this article),  ${\cal F}$ is the family, $\lpn$, of lattice polytopes.
A functional $\oz: \lpn \to \R$ is called translation invariant if $\oz(P+x)=\oz(P)$ for $x\in\Z^n$ and $P\in\lpn$. It is $\slnz$ invariant if $\oz(\phi P)=\oz(P)$ for $\phi\in\slnz$ and $P\in\lpn$,  where $\slnz$ is the special linear group over $\Z$, that is, the group of invertible $n\times n$ matrices with integer coefficients and determinant $1$. We remark that Betke \& Kneser formulated their theorem for unimodularly invariant valuations (that is, also admitting matrices with determinant $-1$) but that their proof also establishes the following result.

\begin{theo}[Betke \& Kneser \cite{Betke:Kneser}]\label{uli}
A functional $\,\oz\!:\lpn\to \R$ is  an $\,\slnz$ and translation invariant valuation if and only if  there exist constants $c_0$, $c_1,\dots, c_n\in\R$ such that 
$$\oz(P) =  c_0 \,L_0(P)+\dots+c_n\,L_n(P)$$
for every $P\in\lpn$. 
\end{theo}

\noindent
Here $L_0(P),\dots,L_n(P)$ are the Ehrhart functionals of $P\in\lpn$, that is, the coefficients of the Ehrhart polynomial (see Section \ref{secdiscSteiner} for the definition). 

\goodbreak

An operator $\oZ:{\cal F}\to \ckn$  is called a Minkowski valuation if $\oZ$ satisfies (\ref{valdef}) and addition on $\ckn$  is Minkowski addition; that is, $$K+L=\{x+y : x\in K,\, y\in L\}.$$  
An operator $\oZ:{\cal F}\to \ckn$  is called $\sln$ equivariant if $\oZ(\phi P)= \phi \oZ P$ for $\phi\in\sln$ and $P\in{\cal F}$. Define $\slnz$ equivariance of operators on $\lpn$ analogously.
In recent years, $\sln$ equivariant operators on convex bodies and the associated inequalities have attracted increased interest (see, for example, \cite{Campi:Gronchi,Haberl:Schuster1,LYZ2000,LYZ04,LZ1997,LYZ2010a}).
For valuations $\oZ:\cpn \to \ckn$ that are $\sln$ equivariant and translation invariant, a complete classification has been established.  Let $n\ge 2$.

\begin{theo}[\!\!\cite{Ludwig:Minkowski}]\label{equi_convex}
An operator $\,\oZ:\cpn\to \ckn$  is an  $\,\sln$ equi\-variant  and translation invariant  Minkowski valuation if and only if  there exists a constant $c\ge 0$ such that 
$$\oZ P=c(P -P)$$
for every $P\in\cpn$. 
\end{theo}

\noindent
The operator $P \mapsto P-P =\{x-y: x,y\in P\}$ assigns to $P$ its difference body.  We remark that no complete  analogue of Hadwiger's theorem for Minkowski valuations (that is, no complete classification of rotation equivariant and translation invariant Minkowski valuations) has been established. It follows from, for example, \cite{Schuster09} that the set of such valuations does not depend on only finitely many parameters.

The  aim of this article is to classify Minkowski valuations on lattice polytopes. The following result is an analogue of Theorem \ref{equi_convex}. Let $n\ge 2$. 

\begin{theo} \label{equi}
An operator $\,\oZ: \lpn\to \ckn$ is an $\,\slnz$ equivariant and translation invariant  Minkowski valuation if and only if there exist  constants $a$, $b\geq0$ such that 
$$
\oZ P=a(P-\dst(P))+b({-P}+\dst(P))
$$
for every $P\in\lpn$.
\end{theo}

\noindent
Here for a lattice polytope $P$, the point $\dst(P)$ is its discrete Steiner point that is introduced in this paper. 

The discrete Steiner point is defined in Section \ref{secdiscSteiner} as the one-homogeneous part of the Ehrhart expansion of the discrete moment vector
$$\ell(P)= \sum_{x\in P\cap \Z^n}\! x.$$ 
That such an expansion exists follows from results by McMullen \cite{McMullen77}.
The discrete moment vector plays for $\slnz$ equivariant vector-valued valuations on $\lpn$ a role similar to that of the moment vector
$$m_{n+1}(K)=\int_K x\dd x$$
for rigid motion equivariant valuations on $\ckn$.

The discrete Steiner point is characterized in the following result, where $\oz: \lpn\to \R^n$ is called translation equi\-variant if $\oz(P+x)= \oz(P)+x$ for $x\in\Z^n$ and  $P\in\lpn$.

\begin{theo} \label{dst}
A function $\,\oz: \lpn\to \R^n$ is an $\,\slnz$ and translation equi\-variant valuation if and only if 
$\,\oz=\dst$.
\end{theo}

\noindent
Theorem \ref{dst} corresponds to the following characterization  of the classical Steiner point $m_1$, which is the one-homogeneous part of the Steiner expansion of the moment vector  (see Section \ref{prelim} for the definition).

\begin{theo}[Schneider  \cite{Schneider72b}] \label{st}
A function $\,\oz: \ckn\to \R^n$ is a continuous and rigid motion equivariant valuation if and only if 
$\,\oz=m_1$.
\end{theo}

A function $\oz:\lpn \to\R^n$ is called additive if $\oz(P+Q)= \oz(P) +\oz(Q)$ for $P,Q\in \lpn$. The discrete Steiner point is also characterized in the following result.

\begin{theo} \label{dst_add}
A function $\,\oz: \lpn\to \R^n$ is $\,\slnz$ and translation equi\-variant  and additive if and only if 
$\,\oz=\dst$.
\end{theo}

\noindent
Theorem \ref{dst_add} corresponds to the following characterization  of the classical Steiner point.

\begin{theo}[Schneider  \cite{Schneider:steiner}] \label{st_add}
A function $\,\oz: \ckn\to \R^n$ is  continuous, rigid motion equivariant  and additive if and only if 
$\,\oz=m_1$.
\end{theo}

\goodbreak
For operators mapping $\lpn$ to $\lpn$, we obtain the following result.
Write $\lcm$ for the least common multiple and let $n\ge 2$. 

\begin{theo} \label{equii}
An operator $\,\oZ: \lpn\to \lpn$ is an $\,\slnz$ equivariant and translation invariant  Minkowski valuation if and only if there exist  integers $a$, $b\geq0$ with $b-a \in\lcm(2,\dots, n+1)\,\Z$ such that 
$$
\oZ P=a(P-\dst(P))+b({-P}+\dst(P))
$$
for every $P\in\lpn$.
\end{theo}

An operator $\oZ:{\cal F} \to \ckn$ is called 
$\sln$ contravariant if $\oZ(\phi P)= \phi ^{-t}\oZ P$ for $\phi\in\sln$ and $P\in{\cal F}$, where $\phi^{-t}$ is the inverse of the transpose of $\phi$. Define $\slnz$ contravariance of operators on $\lpn$ analogously. 
In recent years, $\sln$ contravariant operators on convex bodies and the associated inequalities have attracted increased interest (see, for example, \cite{Haberl:Schuster1,LYZ2000, LYZ2010b,  Boeroeczky2013}).
For $\sln$ contravariant Minkowski valuations on $\cpn$, a complete classification has been established.  Let $n\ge 2$. 

\begin{theo}[\!\!\cite{Ludwig:Minkowski}]
An operator $\,\oZ:\cpn\to \ckn$  is an  $\,\sln$ contra\-variant  and translation invariant  Minkowski valuation if and only if  there exists a constant $c\ge 0$ such that 
$$\oZ P=c \op P$$
for every $P\in\cpn$. 
\end{theo}

\noindent
Here $\op P$ is the projection body of $P$ (see Section \ref{prelim} for the definition).

For operators on lattice polytopes, we obtain the following result.

\begin{theo} \label{contra}
(i) For $n=2$, an operator $\,\oZ:\lpt\to\ckt$ is an $\,\sltz$ contravariant and translation invariant  Minkowski valuation if and only if there exist constants $a,b\geq 0$ such that 
$$
\oZ P=a\,\rho_{\pi/2}(P-\dst(P))+b\,\rho_{\pi/2}({-P}+\dst(P))
$$
for every $P\in\lpt$.\\[2pt]
\noindent
(ii) For $n\geq 3$, an operator $\,\oZ:\lpn\to\ckn$ is an $\,\slnz$ contravariant and translation invariant  Minkowski valuation if and only if  there exists a constant $c\geq 0$ such that
$$\oZ P=c \op P$$ for every $P\in\lpn$.
\end{theo}

\noindent
Here $\rho_{\pi/2}$ denotes the rotation by an angle $\pi/2$ in $\R^2$. Note that for $n=2$, the projection body is obtained from the difference body by applying $\rho_{\pi/2}$.

\goodbreak
For operators mapping $\lpn$ to $\lpn$, we obtain the following result.

\begin{theo} \label{contrai}
(i) For $n=2$, an operator $\,\oZ:\lpt\to\lpt$ is an $\,\sltz$ contravariant and translation invariant  Minkowski valuation if and only if there exist integers $a,b\geq 0$  with $b-a\in 6\,\Z$ such that 
$$
\oZ P=a\,\rho_{\pi/2}(P-\dst(P))+b\,\rho_{\pi/2}({-P}+\dst(P))
$$
for every $P\in\lpt$.\\[2pt]
\noindent
(ii) For $n\geq 3$, an operator $\,\oZ:\lpn\to\lpn$ is an $\,\slnz$ contravariant and translation invariant  Minkowski valuation if and only if there exists a constant $c\geq 0$ with $c\in (n-1)!\,\Z$ such that
$$\oZ P=c \op P$$ for every $P\in\lpn$.
\end{theo}

\section{Preliminaries}\label{prelim}

We collect notation and results on convex bodies, valuations and lattice polytopes. General references are Schneider \cite{Schneider:CB2}, Gruber \cite{Gruber}, Barvinok \cite{Barvinok2008} and Beck \& Robins \cite{BeckRobins}.

Every convex body $K\in \ckn$ is determined by its support function,
$$h(K,v) = \max\{v\cdot x: x\in K\}$$
for $v\in\R^n$, where $v\cdot x$ is the inner product of $v,x\in \R^n$. 
Note that for $v\in\R^n$ we have
$$h(K+L,v) = h(K,v) +h(L,v).$$
Support functions of convex bodies are sublinear, that is, they are convex and positively homogeneous of degree 1,
and every sublinear function is the support function of a convex body in $\R^n$.

For $M\subset\R^n$, we denote the affine hull  by $\aff M$ and the dimension (that is, the dimension of $\aff M$) by $\dim(M)$. Define the centroid of a $k$-dimensional set $M$ with positive $k$-dimensional volume $V_k(M)$ by
$$\cen(M)=\frac 1{V_k(M)} \int_M x\dd H^k(x),$$
where $H^{k}$ is the $k$-dimensional Hausdorff measure.
We denote the convex hull of $x_1,\dots, x_k\in\R^n$  by $[x_1, \ldots, x_k]$.

\subsection{Minkowski summands}

Understanding the structure of summands is critical to our argument. A convex body $L$ is a summand of a convex body $K$ if there exists a convex body $M$ such that $K=L+M$.

For $L\in\ckn$ and $v\in\R^n\backslash\{o\}$, we define the face of $L$ having $v$ as one of its normal vectors by
$$
F(L,v)=\{x\in L:\,  v\cdot x=h(L,v)\}.
$$
It follows that if  $L, M \in\ckn$ and $s,t\geq 0$, then
\begin{equation}
\label{sum-face}
F(s\, L+t\, M,v) =s\, F(L,v)+t\, F(M,v).
\end{equation}
Note that if $K=L+M$ is a polytope, then so are $L$ and $M$.  Also note that the only summands of a simplex $S$ are translates of  $t\,S$ with $t\in[0,1]$ and that a summand of a direct sum of two convex bodies is the direct sum of summands of these bodies  (see \cite[Section 3.2]{Schneider:CB2}). Combined with (\ref{sum-face}) this implies the following.

\begin{lemma}
\label{summand}
Let $S$ be a simplex and $R$ a convex body with $\dim(S+R)=\dim(S)+\dim(R)$.
If a convex body $L$ is a summand of $S+R$, then there exist $t\in[0,1]$ and $R'\subset \aff R$ such that
$L$  is a translate of $\,t\,S+R'$.
\end{lemma}

\subsection{Projection bodies}

For $u\in\sn$ (where $\sn$ is the $(n-1)$-dimensional unit sphere), let $\pi_u$ denote the ortho\-gonal projection to the subspace orthogonal to $u$.
For $K\in\ckn$, the projection body $\op K$ is defined by
$$h(\op K, u)=\vert \pi_u K\vert $$
for $u\in\sn$, where $\vert\cdot\vert$ denotes $(n-1)$-dimensional volume.
Note that if $P$ is an $n$-dimensional polytope in $\R^n$ with facets (that is, $(n-1)$-dimensional faces) $F_1,\dots,F_m$, and corresponding facet normals (that is, exterior unit normals)
$u_1,\dots,u_m$, then
\begin{equation}
\label{projbodyP}
\op P=\frac12\sum_{i=1}^m|F_i|\,{[-u_i,u_i]}.
\end{equation}
The Minkowski relation states that
\begin{equation}
\label{Minkrel}
\sum_{i=1}^m|F_i|\,u_i=o.
\end{equation}
See \cite{Gardner} for more information on projection bodies.

\subsection{The Steiner point}

The intrinsic volumes that are characterized in Hadwiger's theorem are the coefficients of the Steiner polynomial, that is,
$$V_n(K+s \,B^n)= \sum_{j=0}^n s^{n-j} v_{n-j} V_j(K),$$
where $B^n$ is the $n$-dimensional Euclidean unit ball and $v_j$ is the $j$-dimensional volume of the $j$-dimensional Euclidean unit ball. The corresponding expansion for the moment vector is
$$m_{n+1}(K+s \,B^n)= \sum_{j=1}^{n+1} s^{n+1-j} v_{n+1-j} \,m_j(K).$$
The Steiner point, $m_1(K)$, can also be represented as
$$m_1(K) = \frac1{v_n} \int_{\sn} u\, h(K,u)\dd H^{n-1}(u).$$
For more information on Steiner points, see \cite[Section 5.4]{Schneider:CB2}.

\subsection{The inclusion-exclusion principle}

Betke (unpublished) and McMullen  \cite{McMullen09} extended (\ref{valdef}) to an inclusion-exclusion principle.
Let $\group$ be an abelian group.

\begin{theo}[McMullen \cite{McMullen09}]
\label{incl-ex0}
If  $\,\oz:\lpn \to \group$ is a valuation, then  for lattice polytopes $P_1,\dots,P_m$, 
$$
\oz(P_1\cup \dots\cup P_m)=
\sum_{\genfrac{}{}{0pt}{}{1\leq i_1<\dots<i_k\leq m}
{1\leq k\leq m}}
(-1)^{k-1} \oz(P_{i_1}\cap \dots\cap P_{i_k})
$$
whenever $P_1\cup \dots \cup P_m$ and 
all intersections of the form
$P_{i_1}\cap \dots\cap P_{i_k}$  are lattice polytopes.
\end{theo}

The inclusion-exclusion formula is actually needed for cell decompositions in this paper. We call a dissection of  the $n$-dimensional lattice polytope $Q$ into $n$-dimensional lattice polytopes $P_1,\dots,P_m$ a cell decomposition if  $P_i\cap P_j$ is either empty or a common face of $P_i$ and $P_j$ for any $1\leq i< j\leq m$. 
The faces  of the cell decomposition are the faces of all $P_i$ for $i=1,\dots,m$. Let $\Int Q$ denote the interior of $Q$. 

\begin{coro}
\label{incl-ex}
If  $\,\oz:\lpn \to \group$ is a valuation and $Q$ an $n$-dimensional lattice polytope, then
$$
\oz( Q)=
\sum_{\genfrac{}{}{0pt}{}{F\in{\cal F}}  {F\cap \Int Q \ne \emptyset}}
(-1)^{n- \dim (F)}\oz (F),
$$
where ${\cal F}$ is the set of  faces of a cell decomposition of $Q$ into lattice polytopes. 
\end{coro}

\begin{proof}
Let $\mathbf{1}_P$ be the characteristic function of $P\in\lpn$ and  $\clpn$  the additive abelian group generated by characteristics functions of lattice polytopes. McMullen \cite[Theorem~8.1(c)]{McMullen09} established the following form of the  inclusion-exclusion principle. For any valuation $\,\oz:\lpn \to \group$ there exists a homomorphism $[\oz]:\clpn\to \group$ such that
$\oz(P)=[\oz](\mathbf{1}_P)$ for $P\in\lpn$.

Hence  it suffices to show that
\begin{equation}\label{euler}
\mathbf{1}_Q= \sum_{\genfrac{}{}{0pt}{}{F\in{\cal F}}  {F\cap \Int Q \ne \emptyset}} (-1)^{n- \dim (F)}\mathbf{1}_ F.
\end{equation}
Clearly, (\ref{euler}) is true on the complement of $Q$. For $x\in Q$, let $\cS_x$ be the set of faces of $\cF$ that have non-empty intersection with $\Int Q$ and  contain $x$ and  let $\cL_x$ be the boundary complex of the set underlying $\cS_x$. 
For a family of polytopes $\cG$, define $\chi(\cG)$ as the number of even dimensional polytopes minus the number of odd dimensional polytopes in $\cG$. For the cell complex $\cS_x\cup \cL_x$, we obtain the Euler characteristic and $\chi(\cS_x) +\chi(\cL_x)=1$ since the underlying set is homeo\-morphic to an $n$-dimensional ball. This also implies $\chi(\cL_x)=1-(-1)^n$. Hence 
$$\sum_{F\in\cS_x} (-1)^{n- \dim (F)}= (-1)^{n} \chi(\cS_x)=(-1)^{n}( 1- \chi(\cL_x)),$$
which proves (\ref{euler}).
\end{proof}

\subsection{Triangulations}

Write $e_1,\dots,e_n$ for the standard orthonormal basis of $\R^n$, which generates $\Z^n$, and write $o$ for the origin. Define
$T_0=\{o\}$ and $T_i=[o,e_1,\dots,e_i]$ for $i=1,\dots,n$. We call a lattice simplex basic if it is obtained from $T_i$ for some $i=0,\dots, n$ by a map from $\slnz$ followed by a translation.

In addition, let $[0,1]^{i}=[o,e_1]+\dots+[o,e_i]$ be the standard 
$i$-dimensional unit cube. One of the main ideas in this paper is to  relate $\oZ T_n$ and $\oZ [0,1]^{n}$ for a Minkowski valuation $\oZ$ on $\lpn$. In order to do that, we write $R_n$ for the convex hull of all vertices of $[0,1]^{n}$ but $o$. Hence $R_n\cup T_n=[0,1]^{n}$ and $R_n\cap T_n=[e_1,\dots,e_n]$. Since $\oZ$ is a valuation, we get
\begin{equation}
\label{TQW}
\oZ[0,1]^{n} +\oZ [e_1,\dots,e_n]=\oZ T_n+\oZ R_n.
\end{equation}

In the case of $\slnz$ equivariant and translation invariant Minkowski valuations, we also need another specific cell decomposition involving $T_n$. For the prism $\cyl=T_{n-1}+[0,e_n]$, it will be useful to consider a cell decomposition of $\cyl$ into $n$ simplices $S_1,\dots,S_n$. Setting $e_0=o$, we define $S_1=T_n$ and
\begin{equation}
\label{Ttildedissect}
S_i=[e_0+e_n,\dots,e_{i-1}+e_n,e_{i-1},\dots,e_{n-1}] \mbox{ \ for $i=2,\dots,n$}.
\end{equation}
Note that  each $S_i$ is basic and  that $\dim (S_i\cap S_j)=n-1$ for $i<j$ if and only if $j=i+1$
(see, for example,  \cite[Section~2.1]{Hat02}).

We also require the following result (see, for example, \cite[Section 6.3]{DeLoera:Rambau:Santos}).

\begin{lemma}
\label{cube-triang}
There exists a triangulation of $\,[0,1]^{n}$ into $n!$ basic simplices  using only the vertices of the cube such that  $T_n$ is one of these simplices.
\end{lemma}

\subsection{The Betke \& Kneser theorem}

Betke \cite{BetkeHabil} and
Betke \& Kneser  \cite{Betke:Kneser} proved Theorem \ref{uli} by using suitable dissections and complementations of lattice polytopes by lattice simplices.

\begin{prop}[Betke \& Kneser \cite{Betke:Kneser}]
\label{complementation} 
For every lattice polytope $P\in\lpn$ there exist basic simplices $S_1,\dots, S_m$ and integers $k_1,\dots,k_m$ such that
$$
\oz(P)=\sum_{i=1}^m k_i \oz(S_i)
$$
for all valuations $\oz$ on $\lpn$ with values in an abelian group.
\end{prop}

The following statement is a  direct consequence of this proposition.

\begin{coro}
\label{BetkeKneserinvariance}
If $\,\oZ, \oZp\!: \lpn \to \ckn$ are $\,\slnz$ equivariant (or $\,\slnz$ contravariant) and trans\-lation invariant Minkowski valuations such that 
\begin{equation}\label{bkeq}
\oZ T_i=\oZp T_i \,\text{ for }\,i=0,\dots,n,
\end{equation}
then $\oZ=\oZp$ on $\lpn$.
\end{coro}

\begin{proof}
For $v\in\R^n$, the functions $P\mapsto h(\oZ P, v)$ and $P\mapsto h(\oZp P, v)$ are real-valued valuations on $\lpn$.
For $P\in\lpn$, Proposition \ref{complementation} implies that there are basic simplices $S_1,\dots,S_m$ and integers $k_1,\dots,k_m$ such that
$$
h(\oZ P, v)=\sum_{i=1}^m k_i h(\oZ S_i, v) \,\text{ and } \,
h(\oZp P, v)=\sum_{i=1}^m k_i h(\oZp S_i, v).
$$
Since $\oZ$ and $\oZ'$ are $\slnz$ equivariant (or $\slnz$ contravariant) and translation invariant, (\ref{bkeq}) implies that $\oZ=\oZp$ on $\lpn$.
\end{proof}

\subsection{Translation invariant valuations}
We say that a valuation $\oz$ with values in an abelian semigroup  is homogeneous of degree $i\in\N$ 
if $\oz(kP)=k^i \oz (P)$ for $k\in\N$ and $P\in\lpn$, where $\N$ denotes the set of non-negative integers.

McMullen \cite{McMullen77} established the following theorem under the assumption of the inclusion-exclusion principle contained in Theorem~\ref{incl-ex0}, which 
he later established in \cite{McMullen09}.

\begin{theo}[McMullen]
\label{trans-poly}
If\,  $\oz:\lpn \to \R$ is a translation invariant valuation, then $\oz(kP)$ is a polynomial in $k\in\N$ of degree $\dim (P)$ for every $P\in \lpn$.
\end{theo}

\goodbreak
As an application of Theorem~\ref{trans-poly}, we consider Minkowski valuations. The following construction goes back to  \cite{Ludwig:Minkowski}.

\begin{lemma}
\label{Minkowski-poly}
Let $\,\oZ\! :\lpn \to \ckn$ be a translation invariant Minkowski valuation. For $P\in\lpn$, there exists a convex body
$$\oZ_n\! P=\lim_{k\to\infty}\frac{\oZ(kP)}{k^n},$$
and $\,\oZ_n$ is a  Minkowski valuation on $\lpn$, which is homogeneous of degree $n$. If $\,\oZ$ is $\,\slnz$ equivariant or $\,\slnz$ contravariant, then so is $\oZ_n$.
\end{lemma}

\begin{proof} For $v\in\R^n$, the function $P \mapsto h({\oZ P},v)$ is a real-valued valuation on $\lpn$, which is translation invariant as $\oZ$ is translation invariant. By Theorem~\ref{trans-poly}, there exist coefficients
$c_i(P,v)\in\R$, $i=0,\dots,n$,  for $v\in\R^n$ and $P\in \lpn$ such that
$$
h(\oZ(kP),v)=\sum_{i=0}^nc_i(P,v)k^i \mbox{ \ for $k\in\N$}.
$$
Hence the limit $c_n(P,v)=\lim_{k\to\infty}h({\oZ(kP)},v)/k^n$  exists for $v\in \R^n$ and $c_n(P,\cdot)$ is a sublinear function on $\R^n$. Therefore $c_n(P,\cdot)$ is the support function of a convex body, which we call $\oZ_n\! P$.
Since $\oZ$ is a Minkowski valuation, so is $Z_n$. In addition, for fixed $v$, the function $P \mapsto c_n(P,v)$ is homogeneous of degree $n$ in $P$ by Theorem~\ref{trans-poly}. Thus the same holds for $Z_n$.  The equivariance follows immediately from the definition. \end{proof}

\subsection{Transforming into a regular simplex}
\label{secn-2}

Because of Corollary \ref{BetkeKneserinvariance}, we concentrate on determining $\oZ T_n$ in the proof of Theorems \ref{equi} and \ref{contra}. We will make extensive use of the symmetries of $\oZ T_n$.
 
Let $\gln$ denote the group of general linear transformations on $\R^n$.  We write $T=T_n$ and set $\bar T= T-\cen(T)$, where $\cen(T)$ is the centroid of $T$.  We  fix a transformation $\alpha\in\gln$
such that $\alpha \bar T$ is the regular simplex $T_*$ of circumradius one, 
\begin{equation}\label{alpha}
\alpha \bar T=[v_0,\dots,v_n]=T_*,  
\end{equation}
where $v_0= -\cen( \alpha T)$ and $v_i=v_0+\alpha e_i$ for $i=1,\dots,n$. 
Let $\sym$ denote the group of orientation preserving isometries of the regular simplex $T_*$.
\goodbreak

Note that
\begin{equation}
\label{vpq0}
v_i \cdot v_j=\left\{
\begin{array}{rl}
1&\mbox{ if $i=j$,}\\
-\frac{1}n&\mbox{ if $i\neq j$}.
\end{array}\right.
\end{equation}
We set
$$
w_m=v_0+\dots+v_m\mbox{ \ for $m=0,\dots,n-1$}
$$ 
and obtain
\begin{equation}\label{facets1}
\begin{array}{rlc}
F(T_*,w_m)=[v_0, \dots,v_m], \\[6pt]
F(T_*,-w_m)=[v_{m+1},\dots,v_n],
\end{array}
\end{equation}
and for $m=1, \dots, n-1$,
\begin{equation}\label{facets2}
\begin{array}{rlc}
F([v_1,\dots,v_n],w_m)&=&[v_1,\dots,v_m], \\[6pt]
F([v_1,\dots,v_n],-w_m) &=&[v_{m+1},\dots,v_n].
\end{array}
\end{equation}
Note that all faces of $T_*$ are obtained as image of $F(T_*, w_m)$ by suitable maps from $\sym$ for all $m=0,\dots,n-1$.

In the equivariant case, generalized difference bodies are important for us.
A facet normal of the polytope $T_*-T_*$ is a positive multiple of 
$\sum_{i\in I}v_i$ where $I$ is a proper subset of $\{0,\dots,n\}$. Because of (\ref{sum-face}), the same holds true for the facet normals of $s\,T_*-t\,T_*$ with $s,t\ge0$. Since (in any dimension)
$$s\,T_*-t\,T_* = [s\,v_i -t\,v_j: i,j=0,\dots, n, i\ne j],$$
we get
\begin{equation}\label{Xist}
s\,T_*-t\,T_*= [\rho\, F(s\,T_*-t\,T_*, w_m): \rho\in \sym]
\end{equation}
for each $m=0,\dots,n-1$.
Indeed, the right side is clearly contained in $s\, T_*-t\,T_*$ and it follows from (\ref{facets1}) that each $s\,v_i -t\,v_j$ for $i\ne j$ is contained in the right side.

\section{The discrete Steiner point}
\label{secdiscSteiner}

For $P\in\lpn$,  let  $L(P)$ denote the number of lattice points in $P$, that is,
\begin{equation}\label{lpe}
L(P)=\sum_{x\in P\cap\Z^n}1.
\end{equation}
The function $L:\lpn \to \Z$ is a valuation that is invariant with respect to unimodular linear transformations.
In addition, if $z\in\Z^n$, then
\begin{equation*}
L(P+z)=L(P),
\end{equation*}
that is, $L$ is translation invariant. We call a function that is invariant with respect to unimodular linear transformations and translations by integer vectors unimodularly invariant.

\goodbreak
Ehrhart  \cite{Ehrhart62}  established the following result. 

\begin{theo}[Ehrhart]\label{Ehrhart}
There exist $L_i: \lpn \to \Q$ \,for $i=0,\ldots, n$ such that
$$L(kP)=\sum_{i=0}^{n}L_i(P)k^i$$ 
for $k\in\N$ and $P\in \lpn$. For each $i$, the functional $L_i$ is a unimodularly invariant valuation which is homogeneous of degree $i$.
\end{theo}

\noindent Note that $L_n$ is the $n$-dimensional volume and $L_0$ is the Euler characteristic.

\goodbreak
A special case of a more general result by 
McMullen \cite[Theorem~6]{McMullen77} implies 

\begin{theo}[McMullen]
\label{lattice-polynomial} If $\,P_1,\dots, P_m\in\lpn$, then the number of lattice points $L(k_1 P_1 +\dots +k_m P_m)$ is a polynomial in $k_1,\dots, k_m\in\N$.
\end{theo}

\noindent
From this, the following 
 analogue of Remark~6.3.3 in Schneider \cite{Schneider:CB2}
is obtained.

\begin{coro}
\label{polynomiality-coefficient1}
The functional $L_1: \lpn \to\Q$ is Minkowski  additive.
\end{coro}
\begin{proof}
For $P,Q\in \lpn$ and $k,l\in\N$, by Theorem \ref{lattice-polynomial} we see that $L(k\, P+l\, Q)$ is a polynomial in $k$ and $l$. Considering the expression first as a function of $l$ when $k=0$ and second as a function of $k$ when $l=0$, we deduce that the linear term in $L(k\, P+l\, Q)$ is $L_1(P)\,k+L_1(Q)\, l$. In particular, the linear term
in the one variable polynomial $L(k\, P+k\, Q)$ is $L_1(P)\,k+L_1(Q)\,k$ on the one hand and 
by Theorem \ref{Ehrhart} we get $L_1(P+Q)\,k$ on the other hand. Hence $L_1(P+Q)=L_1(P)+L_1(Q)$.
\end{proof}

\goodbreak
In analogy to (\ref{lpe}), for $P\in\lpn$, we define the discrete moment vector  by
$$
\ell(P)=\sum_{x\in P\cap\Z^n}x.
$$
The discrete moment vector $\ell:\lpn \to \Z^n$ is a valuation that is equivariant with respect to unimodular linear transformations.
In addition, if $z\in\Z^n$, then
\begin{equation}
\label{moment-trans}
\ell(P+z)=\ell(P)+L(P)z.
\end{equation}
In particular, $\ell$ is not translation invariant or equivariant. 
In the terminology of \cite{McMullen77}, $\ell$ is an extended $\Z^n$-valuation. 

\goodbreak
As a special case of a more general result by McMullen \cite[Theorem 14]{McMullen77} we obtain the following 
 result.

\begin{theo}[McMullen]
\label{moment-poly}
There exist $\ell_i: \lpn \to \Q^n$ for $i=1,\ldots, n+1$ such that
$$\ell(kP)=\sum_{i=1}^{n+1}\ell_i(P)k^i$$ 
for $k\in\N$ and $P\in \lpn$. For each $i$, the function $\ell_i$ is a valuation which is
equivariant with respect to unimodular linear transformations and
homogeneous of degree $i$. 
\end{theo}

\noindent
We call $\ell_1(P)$ the discrete Steiner point of the lattice polytope $P$. 

\goodbreak

A special case of a more general result by
McMullen \cite[Theorem~14]{McMullen77} implies the following result.

\begin{theo}[McMullen]\label{lattice-polyv} 
If $\,P_1,\dots, P_m\in\lpn$, then the discrete moment vector $\ell(k_1 P_1 +\dots +k_m P_m)$ is a polynomial in $k_1,\dots, k_m\in\N$.
\end{theo}

\noindent
From this, we deduce as in Corollary \ref{polynomiality-coefficient1} the following result.

\begin{coro}\label{dst_additive}
The functional $\dst: \lpn \to\Q^n$ is additive.
\end{coro}

In the next proposition, we collect some properties of the functional $\dst$. We require the following lemma.

\begin{lemma}\label{cTm:lemma}
If $\,\oz:\lpn\to \R^n$ is an $\,\slnz$ equivariant valuation, then 
\begin{equation*}
\label{centrans}
\oz\big((m+1)T_m-\cen((m+1)T_m)\big)=o.
\end{equation*}
\end{lemma}

\begin{proof} First, let $m=n$. 
Note that $\cen((n+1)T_n)\in\Z^n$.
Set 
$$S_n=(n+1)T_n-\cen((n+1)T_n).$$ 
Since $\oz$ is $\slnz$ equivariant, we obtain  from
$\alpha^{-1} \rho \alpha S_n= S_n$ (with $\alpha$ defined in (\ref{alpha})) that
$$\alpha \oz(S_n)= \rho  \alpha \oz(S_n)$$
for all $\rho\in\sym$. Thus the statement holds for $m=n$. The lower dimensional case follows by considering the statement in an appropriate subspace.
\end{proof}

\begin{prop}
\label{disc-steiner}
The functional $\dst:\lpn \to \Q^n$ is an $\,\slnz$ and translation equivariant valuation.
If $P\in\lpn$ is a basic simplex or centrally symmetric, then $\dst(P)=\cen(P)$.
\end{prop}

\begin{proof}
That $\dst:\lpn\to \Q^n$ is an $\slnz$ equivariant valuation is part of Theorem~\ref{moment-poly}. That $\dst$ is translation equivariant follows from Theorem \ref{moment-poly} and (\ref{moment-trans}). 

If $T$ is an $m$-dimensional basic simplex, then it follows from
Lemma \ref{cTm:lemma} and the trans\-lation equivariance that $\dst((m+1)T)=\cen((m+1)T)$. As both $\dst$ and the centroid are homogeneous of degree one, we conclude $\dst(T)=\cen(T)$.

If $P\in \lpn$ is centrally symmetric, then $\cen(P)$ is the center of symmetry of $P$. 
If $x_0$ is a vertex of $P$, then its image $x_1$  by the reflection through $\cen(P)$ is also a vertex. Thus $x_0,x_1\in\Z^n$ and $\cen(P)=\frac12(x_0+x_1)$. 
The unimodular map $\phi$ defined by $z\mapsto -z+x_0+x_1$ is the reflection through $\cen(P)$ and
its only fixed point is $\cen(P)$. Since both $\cen(P)$ and $\dst(P)$ are fixed points of $\phi$, we conclude that $\dst(P)=\cen(P)$.
\end{proof}

\subsection{Proof of Theorem \ref{dst}}

That $\dst:\lpn\to \R^n$ is an $\slnz$ and translation equivariant  valuation is part of Proposition~\ref{disc-steiner}. So the following proposition concludes the proof of the theorem. Let $n\ge 2$.

\begin{prop}
If $\,\oz:\lpn\to \R^n$ is an $\slnz$ and translation equi\-variant valuation, then $\oz=\dst$.
\end{prop}

\begin{proof} 
Define $\ow: \lpn \to \R^n$ by $\ow(P)=\oz (P)-\dst(P)$. Note that $\ow$ is an $\slnz$  equivariant and translation invariant valuation. Applying Theorem~\ref{trans-poly} to $\ow$ shows that
if $P\in \lpn$ and $k\in\N$, then
$$
\ow (kP)=\sum_{i=0}^n \ow_i(P)k^i
$$
where for each $i$, the function $\ow_i:\lpn \to \R^n$ is an $\slnz$  equivariant and translation invariant valuation which is homogeneous of degree $i$.

Lemma \ref{cTm:lemma} applied with $\ow=\ow_i$, the $\slnz$ equivariance and translation invariance of $\ow_i$ imply that 
$$\ow_i((m+1) T_m)=\ow_i((m+1)T_m-\cen((m+1)T_m))=o$$ 
for $i=0,\dots,n$. Since $\ow_i$ 
is homogeneous of degree $i$, we obtain $\ow_i(T_m)=o$ for $i=0,\dots,n$. Thus Corollary~\ref{BetkeKneserinvariance} implies that  $\ow(P)=o$
for $P\in \lpn$. In particular, $\oz (P)=\dst(P)$ for any $P\in \lpn$.
\end{proof}

\subsection{Proof of Theorem \ref{dst_add}}

Since any additive function on $\lpn$ is a valuation, Corollary \ref{dst_additive} implies the statement of the theorem.

\section{Contravariant valuations}

In this section, we first prove Theorem \ref{contra}, that is,  we prove that for every $\sltz$ contra\-variant and translation invariant Minkowski valuation $\oZ$ on $\lpt$, 
there are $a,b\ge 0$ such that $\oZ P=a\,\rho_{\pi/2}(P-\dst(P))+b\,\rho_{\pi/2}({-P}+\dst(P))$ for every $P\in\lpt$ and 
we prove for $n\ge3$ that for every $\slnz$ contra\-variant and translation invariant Minkowski valuation $\oZ$ on $\lpn$, 
there is $c\ge 0$ such that $\oZ=c\op$. Second, we prove Theorem \ref{contrai}.

Note that a simple consequence of the symmetry properties of $T_n$ and the $\slnz$ contravariance of $\oZ$ is the following result.

\begin{lemma}
\label{dimT}
Let $\oZ: \lpn \to \ckn$ be an $\,\slnz$ contravariant  and translation invariant Minkowski valuation.
If
$\,\oZ T_n \neq\{o\}$, then $o\in\Int (\oZ T_n)$.
\end{lemma}

\subsection{Lower dimensional polytopes}
\label{secTn-symmetries}

The next lemma was proved in \cite{Haberl_sln, Ludwig:projection} for $\sln$ equivariant (and homo\-geneous) valuations  on $\cpn$.

\goodbreak
\begin{lemma}
\label{lowdim-contra}
Let $\oZ:\lpn\to\ckn$ be an $\,\slnz$ contravariant  and translation invariant Minkowski valuation and let $P\in\lpn$.
\begin{description}
\item{(i)} If $\,\dim (P)\leq n-2$, then $\oZ P=\{o\}$.
\item{(ii)} There exists  $c\geq 0$ (depending on $\oZ$) such that 
if $\dim (P)= n-1$ and $w$ is a unit normal to $\aff P$, then $\oZ P=c\,|P|\,[-w,w]$.
\end{description}
\end{lemma}

\goodbreak
\begin{proof} By translation invariance and $\slnz$ contravariance, we may assume that $\Span P=\Span\{e_1,\dots,e_d\}$, where $d=\dim (P)\leq n-1$.

First we claim that
\begin{equation}
\label{inRej}
\oZ P\subset \Span\{e_j\} \mbox{ \ for $j=d+1,\dots,n$}.
\end{equation}
To simplify the notation, let $j=n$ in (\ref{inRej}).

For $j\in\Z$ and $k\in\{1,\ldots,n-1\}$, we define
$\phi_{jk}\in \slnz$ by $\phi_{jk}e_i=e_i$ if $i\neq n$, and $\phi_{jk}e_n=je_k+e_n$. It follows that $\phi_{jk}P=P$. If we have $x=\sum_{i=1}^nt_ie_i\in \oZ P$, then
\begin{equation}
\label{sqAt}
\phi_{jk}^{-t}x=(t_n-j\, t_k)e_n+\sum_{i=1}^{n-1}t_ie_i.
\end{equation}
Since $\phi^{-t}_{jk}\oZ P=\oZ P$, the vector $\phi_{jk}^{-t}x$ is contained in a bounded set. Since $k\in\{1,\dots,n-1\}$ and $j\in\Z$ are arbitrary in (\ref{sqAt}), we conclude that $t_1=\dots=t_{n-1}=0$. Thus (\ref{inRej}) and therefore also (i) are proved.

To prove (ii), we identify $\Span\{e_1,\dots,e_{n-1}\}$ with $\R^{n-1}$. By (\ref{inRej}),
 there exist real $\oz_1(P)\leq \oz_2(P)$
such that $\oZ P=[\oz_1(P),\oz_2(P)]e_n$ for a lattice polytope $P\in\lpm$. In particular, $\oz_1$ and $\oz_2$ are $\slmz$ and translation invariant valuations. Let $a_i=\oz_i(T_{n-1})/|T_{n-1}|$.
Since $\oz_1(S)=\oz_2(S)=0$ by (i) if $S$ is a basic simplex of dimension at most  $(n-2)$, the $(n-1)$-dimensional version of Corollary~\ref{BetkeKneserinvariance} implies that 
$\oz_i(P)=a_i\,|P|$  for $P\in\lpm$. 

To relate $a_1$ and $a_2$ for $n\geq 3$,  we consider
$\phi\in \slnz$ defined by $\phi e_1=e_2$, $\phi e_2=e_1$, $\phi e_n=-e_n$ and  $\phi e_i=e_i$ if $2<i<n$. Then $\phi T_{n-1}=T_{n-1}$ and
$\phi^{-t}=\phi$. Hence $c=a_2=-a_1\ge0$.  If $n=2$, then the $\sltz$ and translation invariance imply
for $\psi\in \sltz$ defined by $\psi e_1=-e_1$ and  $\psi e_2=-e_2$ that
$$
\oZ T_1=\oZ(T_1-e_1)=\oZ(\psi T_1)=\psi \oZ T_1.
$$
Thus again $c=a_2=-a_1\ge0$.
\end{proof}

\goodbreak
Combining Lemma~\ref{lowdim-contra} and the inclusion-exclusion property  leads to the following result.

\begin{coro}
\label{Zincl-excl-contra}
Let $\oZ: \lpn \to \ckn$ be an $\slnz$ contravariant  and translation invariant Minkowski valuation. If $P_1,\dots,P_k\in\lpn$
form a cell decomposition of an $n$-dimensional lattice polytope, then
$$
\oZ(P_1\cup \dots\cup P_k)\,\,\,+\!\!\!
\sum_{\dim (P_i\cap P_j)=n-1}\oZ(P_i\cap  P_j)=\sum_{i=1}^k \oZ P_i.
$$
\end{coro}

\subsection{Simple valuations}

A valuation $\oZ$ on $\lpn$ is called simple, if $\oZ P=\{o\}$ for every lower dimensional $P\in\lpn$.

\begin{lemma}
\label{ZsimpleW}
Let $\oZ:\lpn\to \ckn$ be an $\,\slnz$ contravariant and translation invariant valuation. If $\,\oZ$ is simple, then  $\oZ [0,1]^{n}=\{o\}$.
\end{lemma}
\begin{proof}
First, we consider the case $n=2$ to show the idea. In this case $[0,1]^2$ can be triangulated into $T_2$ and $T_2'=e_1+e_2-T_2$, and
hence
\begin{equation*}
\oZ [0,1]^2=\oZ T_2 +\oZ T_2'.
\end{equation*}
The $\sltz$ contravariance and the translation invariance of $\oZ$ imply that
$$
\oZ T_2'=-\oZ T_2.
$$
Let $\phi\in\sltz$ be defined by $\phi e_1=e_2-e_1$, 
$\phi e_2=-e_1$. 
We have $\phi T_2= T_2-e_1$ and $\phi T_2' = T_2' -2 e_1$.
Hence $\oZ [0,1]^2$  is invariant under
$\phi^{-t}$. In addition,  define $\psi\in\sltz$ by $\psi e_1= -e_2$ and $\psi e_2=e_1$. Note that $\oZ [0,1]^2$ is  invariant under $\psi=\psi^{-t}$. 

Suppose that there exists $x=(x_1,x_2)\in \oZ [0,1]^2\backslash \{o\}$ and seek a  contradiction.
By the $\psi$ invariance of $\oZ [0,1]^2$, we may assume that $x_2\neq 0$. We observe that
for $\vartheta=\psi\circ \phi^{-t}$, we have
$$
\vartheta x=(x_1-x_2,x_2)\in \oZ [0,1]^2.
$$
Since $\oZ [0,1]^2$ is invariant under $\vartheta^k$ for any $k\geq 1$, it follows that the points
$(x_1-kx_2,x_2)\in \oZ [0,1]^2$ for any $k\geq 1$. This contradicts the boundedness of $\oZ [0,1]^2$.

Next, let $n\geq 3$ and let
$$
Q=T_2+\sum_{i=3}^n[o,e_i]
\mbox{ \ and \ } Q'=T_2'+\sum_{i=3}^n[o,e_i].
$$
We define $\eta\in\slnz$ by $\eta e_1=-e_1$, $\eta e_2=-e_2$ and $\eta e_j=e_j$ for $j=3,\dots,n$. 

This map satisfies
$Q'=e_1+ e_2+\eta Q$ and $\eta^{-t}=\eta$. Hence 
\begin{equation}
\label{XiT2}
\oZ Q'=\eta \oZ Q.
\end{equation}
In addition, let $\gamma\in\slnz$ be defined by $\gamma e_1=e_2-e_1$, 
$\gamma e_2=-e_1$, and $\gamma e_j=e_j$ for $j=3,\dots,n$. Thus 
\begin{equation}
\label{Phi12x}
\gamma^{-t}(x_1,x_2,\dots,x_n)=(-x_2,x_1-x_2,x_3,\dots,x_n).
\end{equation}
Since $\gamma Q=Q-e_1$ and $\gamma$ commutes
with $\eta$, it follows from (\ref{XiT2}) that
\begin{equation}
\label{Phi12T2}
\gamma^{-t}\oZ Q =\oZ Q
\mbox{ \ and \ }
\gamma^{-t}\oZ Q' =\oZ Q'.
\end{equation} 
We observe that $Q$ and $Q'$ form a polytopal cell decomposition of $[0,1]^{n}$, 
and hence $\oZ [0,1]^{n}=\oZ Q+\oZ Q'$. We conclude from (\ref{Phi12T2}) that
$$
\gamma^{-t}\oZ [0,1]^{n}=\oZ [0,1]^{n}.
$$

Finally,  suppose that there exists $x=(x_1,\dots,x_n)\in \oZ [0,1]^{n}\backslash \{o\}$, and seek a  contradiction.
For $i,m\in\{1,\dots,n\}$ with $i\ne m$, define $\psi_{im}\in \slnz$ by setting $\psi_{im} e_i =e_m$, $\psi_{im} e_m= -e_i$ and $\psi_{im} e_j= e_j$ for $j\ne i,m$.
By the $\psi_{i,i+1}$ invariance of $[0,1]^{n}$ for $i=1,\dots,n-1$, we may assume that $x_2\neq 0$. We deduce by (\ref{Phi12x}) that
$\vartheta=\psi_{21}\circ \gamma^{-t}$ satisfies
$$
\vartheta x=(x_1-x_2,x_2,x_3,\dots,x_n)\in \oZ [0,1]^{n}.
$$
Since $\oZ [0,1]^{n}$ is invariant under $\vartheta^k$ for any $k\geq 1$, it follows that
$$(x_1-kx_2,x_2,x_3,\dots,x_n)\in \oZ [0,1]^{n}$$ for any $k\geq 1$. This contradicts the boundedness of $\oZ [0,1]^{n}$.
\end{proof}

\subsection{The cube}
\label{seccube-contra}

Let $n\geq 2$ and recall that the constant $c$ was defined in  Lemma~\ref{lowdim-contra}.

\begin{lemma}
\label{cube-contra-expand}
If $\,\oZ: \lpn\to \ckn$ is an \,$\slnz$ contravariant and translation invariant  Minkowski valuation,  then
$$
\oZ(k[0,1]^{n})+c\, (k^n-k^{n-1})\op [0,1]^{n}=k^n\oZ [0,1]^{n} 
$$
for $k\in\N$.
\end{lemma}
\begin{proof}
For $k\geq 1$, we subdivide $k[0,1]^{n}$ into a cell decomposition of $k^n$ unit cubes, and hence all $m$-dimensional faces are unit cubes of dimension $m$. For $i=1,\dots,n$, there exist
$k^{n-1}(k-1)$ facets of the cell decomposition which intersect the interior of $k[0,1]^{n}$ and whose
unit normal vector is $e_i$. 
We deduce from Lemma~\ref{lowdim-contra} and Corollary~\ref{Zincl-excl-contra} that
$$
\oZ(k[0,1]^{n})+(k^n-k^{n-1})\sum_{i=1}^nc\,[-e_i,e_i]=k^n\oZ [0,1]^{n}.
$$
The definition of the projection body, (\ref{projbodyP}), gives $[-c,c]^n=c\op [0,1]^{n}$.
 \end{proof}

\begin{prop}
\label{cube-contra}
If $\,\oZ: \lpn\to \ckn$ is an \,$\slnz$ contravariant and translation invariant  Minkowski valuation,  then 
$\oZ [0,1]^{n}=c\op [0,1]^{n}$ .
\end{prop}
\begin{proof}
Consider the $\slnz$ contravariant and translation invariant  valuation $\oZ_n$ defined in Lemma~\ref{Minkowski-poly}. Since $\oZ_n$ is homogeneous of degree $n$, we deduce from
Lemma~\ref{cube-contra-expand} applied to $\oZ_n$ that $\oZ_n$ is simple. Lemma \ref{ZsimpleW} implies that $\oZ_n [0,1]^{n}=\{o\}$. In particular,
we have $\lim_{k\to\infty}\oZ(k[0,1]^{n})/k^n=\{o\}$. Next we apply Lemma~\ref{cube-contra-expand} to $\oZ$. Dividing both sides by $k^n$ and letting $k\to\infty$ implies that $\oZ [0,1]^{n}=c \op [0,1]^{n}$.
\end{proof}

\subsection{The planar  case}
\label{secplanar}

It is easy to see that
$$
\rho_{\pi/2}\,\phi\,\rho_{-\pi/2}=\phi^{-t}\mbox{ \ for any $\phi\in\slt$}.
$$
As in \cite{Ludwig:Minkowski}, we deduce the following result.

\begin{lemma}
\label{planar-contra-equi}
An operator $\oZ: \lpt \to \ckt$  is \,$\sltz$ equi\-variant if and only if
$\,\rho_{\pi/2}\oZ: \lpt \to \ckt$ is  \,$\sltz$ contravariant.
\end{lemma}

For the next lemma, recall that for given $\oZ$, the constant $c$ was defined in  Lemma~\ref{lowdim-contra}.

\begin{lemma}
\label{ZT20}
If $\,\oZ: \lpt \to \ckt$ is an $\,\sltz$ contravariant and trans\-lation invariant  Minkowski valuation, then
there exist $a,b\geq 0$ such that
$$
\oZ T_2=a\,\rho_{\pi/2}(T_2-\cen(T_2))+b\,\rho_{\pi/2}(-T_2+\cen(T_2))
$$
and  $a+b=2 c$.
\end{lemma}
\begin{proof} 
We dissect $[0,1]^2$ into the triangles $T_2$ and $T_2'=e_1+e_2-T_2$. Since $\oZ$ is a valuation, 
$$
\oZ [0,1]^2 +\oZ (T_2\cap T_2')=\oZ T_2 +\oZ T_2'.
$$
Combining Lemma~\ref{lowdim-contra} and Proposition~\ref{cube-contra} leads to
\begin{equation}
\label{ZT2a}
c\,[-1,1]^2+c\,[-(e_1+e_2),e_1+e_2]=\oZ T_2 +\oZ T_2'.
\end{equation}
Therefore $\oZ T_2$ is a Minkowski summand of the hexagon on the left hand side. If $c=0$, then $\oZ T_2=\{o\}$, and
Lemma~\ref{lowdim-contra} implies that $\oZ$ is simple. Therefore $\oZ P=\{o\}$ for all $P\in{\cal P}^2$.
So let $c>0$. Since $\oZ$ is $\sltz$ contravariant, we have $\oZ T_2'=-\oZ T_2$ and (\ref{ZT2a}) implies that 
$\oZ T_2$ is a two-dimensional polygon whose sides are parallel to $e_1$ or $e_2$ or $e_1+e_2$. 
In addition, let $\phi\in\sltz$ be defined by $\phi e_1=e_2-e_1$ and
$\phi e_2=-e_1$. 
Then $\oZ T_2$ is invariant under 
$\phi ^{-t}$. 
Since $\phi$ permutes  $e_1,-e_1+e_2, -e_2$ on the one hand
and  $-e_1,e_1-e_2, e_2$ on the other hand,  we  have 
$h(\oZ T_2, e_1)= h(\oZ T_2, -e_1+e_2) =h(\oZ T_2, -e_2)$
and
$h(\oZ T_2, -e_1)= h(\oZ T_2, e_1-e_2) =h(\oZ T_2, e_2)$.
Thus it is easy to check that
$$
\oZ T_2=a \,\rho_{\pi/2}(T_2-\cen(T_2))+b\,\rho_{\pi/2}(-T_2+\cen(T_2))
$$
for suitable $a,b\geq 0$.
From (\ref{ZT2a}) we obtain  
$a+b=2c$. 
\end{proof}

\noindent {\bf Proof of Theorem~\ref{contra} in the planar case.}
It follows from Proposition~\ref{disc-steiner} that $P\mapsto P-\dst(P)$ and $P\mapsto -P+\dst(P)$ are 
 $\sltz$ equivariant and translation invariant Minkowski valuations on $\lpt$. We deduce from Lemma~\ref{planar-contra-equi} that $P\mapsto \rho_{\pi/2}(P-\dst(P))$ and 
$P\mapsto \rho_{\pi/2}(-P+\dst(P))$ are
 $\sltz$ contra\-variant and translation invariant  Minkowski valuations.

Since $\dst(T_i)=\cen(T_i)$ for $i=1,2$ by Proposition~\ref{disc-steiner}, combining Lemma~\ref{ZT20},  Lemma~\ref{lowdim-contra} and Corollary~\ref{BetkeKneserinvariance} shows that any  $\sltz$ contravariant and translation invariant  Minkowski valuation $\oZ$ is
of the form
$$
\oZ P=a\,\rho_{\pi/2}(P-\dst(P))+b\,\rho_{\pi/2}(-P+\dst(P)),
$$
where $a,b\geq 0$. \hfill$\qed$

\subsection{Proof of Theorem~\ref{contra} for $\mathbf{n\ge 3}$}

For Minkowski summands, 
we need the following (probably well-known) statement, for which we have not found a reference. Let $n\ge 3$.

\begin{lemma}
\label{zonotope}
Let $v_1,\dots,v_m$ be vectors in $\,\R^n$ such that any $n$ of these vectors are linearly independent. 
If $\,P$ is an $n$-dimensional polytope such that every edge of $P$ is parallel to some $v_i$, 
then $P$ is a translate of $\,\sum_{i=1}^m a_i\,[o,v_i]$ 
with $a_i\geq 0$.
\end{lemma}
\begin{proof} 
A polytope is a zonotope, if all its two-dimensional faces are centrally symmetric (cf.~\cite[Theorem 3.5.2]{Schneider:CB2}). Thus it is sufficient to show that $P$ has centrally symmetric two-dimensional faces.

  We may assume that $P$ has an edge parallel to $v_i$ for every $i=1,\dots,m$. Let
$[x,y]$ be an edge of $P$ parallel to $v_m$ where $y=x+a_mv_m$ for $a_m>0$. We claim that for any vertex $w$ of $\pi_{v_m}P$, there exists a vertex $z$ of $P$ such that
\begin{equation}
\label{edgevk}
\mbox{$\pi_{v_m}z=w$ and $z+a_m v_m$ is a vertex of $P$.}
\end{equation}
Since there is a path of the edge graph of $\pi_{v_m}P$ connecting $\pi_{v_m} x$ and $w$, we may assume that 
$[w,\pi_{v_m} x]$ is an edge of $\pi_{v_m}P$. Let $L$ be the span of $v_m$ and $w-\pi_{v_m}x$.
It follows that $(x+L)\cap P$ is a two-dimensional face of $P$. As no three vectors from $V=\{v_1,\dots,v_m\}$ are linearly dependent, we deduce that for some $i\in\{1,\dots,m-1\}$ we have
$L\cap V=\{v_i,v_m\}$. As the two-dimensional face $F=(x+L)\cap P$ has only edges parallel to $v_i$ and $v_m$, it is
 a parallelogram and hence centrally-symmetric.
\end{proof}

Let $c\geq 0$ be the constant of Lemma~\ref{lowdim-contra}.
First we use the triangulation $S_1,\dots,S_{n!}$ of $[0,1]^{n}$ into basic simplices
given by Lemma~\ref{cube-triang} with $S_1=T_n$. If $\dim (S_i\cap S_j)=n-1$, then 
Lemma~\ref{lowdim-contra} provides a non-zero $p_{ij}\in\R^n$ such that
$\oZ(S_i\cap S_j)=c[-p_{ij},p_{ij}]$. Applying  Corollary~\ref{Zincl-excl-contra} and  Proposition~\ref{cube-contra} to the cell decomposition of $[0,1]^{n}$, we deduce that
\begin{equation}
\label{Wtriang-contra}
[-c,c]^n+\sum_{\dim (S_i\cap S_j)=n-1}c\,[-p_{ij},p_{ij}]=\sum_{i=1}^{n!} \oZ S_i .
\end{equation}
If $c=0$, then (\ref{Wtriang-contra}) implies that $\oZ T_n=\oZ S_1=\{o\}$ and we conclude from Lemma~\ref{lowdim-contra} and
Corollary~\ref{BetkeKneserinvariance} that $\oZ P=\{o\}$ for $P\in\lpn$.

Assume that $c>0$. Hence the left hand side of (\ref{Wtriang-contra}) is full dimensional. Since 
$\oZ S_i=\phi_i^{-t}\oZ T_n$ for   $\phi_i\in \slnz$ with $S_i=\phi_i T_n$, it follows that $\oZ T_n\neq\{o\}$.
We deduce from Lemma~\ref{dimT} that $\dim(\oZ  T_n)=n$.

We consider the decomposition of $[0,1]^{n}$ into $T_n$ and $R_n$ from (\ref{TQW}).
For $w=-e_1-\dots-e_n$, Lemma~\ref{lowdim-contra} implies that
$$
\oZ [e_1,\dots,e_n]=\frac{c}{(n-1)!}\,[-w,w].
$$
Therefore it follows from Proposition~\ref{cube-contra} that (\ref{TQW}) can be written in the form
\begin{equation}
\label{TQW0}
[-c,c]^n+\frac{c}{(n-1)!}\,[-w,w]=\oZ T_n+\oZ R_n.
\end{equation}
We deduce from (\ref{TQW0}) right away that $\oZ T_n$ is a polytope.
Each edge of the left side of (\ref{TQW0}) is parallel to either $w$ or to an $e_i$, $i=1,\dots,n$. As any $n$ of the vectors
$w,e_1,\dots,e_n$ are linearly independent, (\ref{TQW0}) and Lemma~\ref{zonotope} imply that
\begin{equation}\label{a0n}
\oZ T_n \text { is a translate of  }\, a_0 [o, w]+ a_1[ o ,e_1]+\dots+a_n [o, e_n]
\end{equation}
where $a_i\geq 0$ for $i=0,\dots,n$. 

Let $\alpha$ be defined as in (\ref{alpha}) and let $\rho\in\sym$. Note that the map $\alpha^{-1}\rho \alpha \in \slnz$. Hence, by the $\sln$ contravariance of $\oZ$, 
$$
\alpha^{-t} \oZ T_n = \rho \alpha^{-t} \oZ T_n.$$ 
Since this holds for all $\rho\in\sym$, it follows that $o=\cen( \alpha^{-t}\oZ T_n)=\cen(\oZ T_n)$.
Since $\rho\in\sym$ permutes the normal vectors of $T_*$, it permutes $\alpha^{-t}w$, $\alpha^{-t} e_1$, $\dots$, $\alpha^{-t}e_n$ and  we obtain $a_0=\dots=a_n$ in (\ref{a0n}). 
Taking into account (\ref{projbodyP}), we conclude that $\oZ T_n =c_0 \op T_n$ for $c_0>0$.

To determine $c_0$, we deduce from (\ref{Wtriang-contra}) that
$$
[-c,c]^n+\sum_{ \dim(S_i\cap S_j)=n-1}c\,[-p_{ij},p_{ij}]=c_0\, \sum_{i=1}^{n!} \phi_i^{-t}\op T_n.
$$
The $\slnz$ contravariant and translation invariant  Minkowski valuation  $c\op$ also satisfies
 (\ref{Wtriang-contra}). Hence $c_0=c$. Thus  Theorem~\ref{contra} follows from Corollary~\ref{BetkeKneserinvariance}.

\subsection{Proof of Theorem \ref{contrai}}

First, let $n=2$.
Propositions \ref{complementation}  and \ref{disc-steiner} imply that $6\, \dst(P)\in \Z^2$ for all $P\in\lpt$. 
Hence, for integers $a,b\ge 0$ with $b-a\in 6 \,\Z$, the operator $\oZ$ defined by
$$P \mapsto a\,\rho_{\pi/2}(P-\dst(P))+b\,\rho_{\pi/2}({-P}+\dst(P))$$
maps $\lpt$ to $\lpt$. For the reverse direction, let $\oZ:\lpt \to \lpt$ be an $\sltz$ contravariant and translation invariant Minkowski valuation.  By Theorem \ref{contra}, we know that there are $a,b\ge 0$ such that
$$\oZ P= a\,\rho_{\pi/2}(P-\dst(P))+b\,\rho_{\pi/2}({-P}+\dst(P))$$
for every $P\in \lpt$.
Since $\oZ P\in \lpt$ for all $P\in\lpt$, setting $P=T_1$ shows that $a+b\in 2\, \Z$ and setting $P=T_2$ shows that $2 a+b, a+2 b\in 3\,\Z$. Thus $a,b\in \Z$ and $b-a\in 6\,\Z$.

\goodbreak
Next, let $n\ge 3$. By (\ref{projbodyP}) and (\ref{Minkrel}), the projection body of $P\in \lpn$ with facet normals $u_1,\dots, u_m$ and corresponding facets $F_1, \dots, F_m$ is 
$$\sum_{i=1}^m \vert F_i\vert \,[o,u_i].$$
Since every facet can be triangulated and the $(n-1)$-dimensional volume of an $(n-1)$-dimensional lattice simplex is an integer multiple of $1/(n-1)!$, we have $P \mapsto c \op P$ with $c\in (n-1)!\, \Z$ is an operator that maps $\lpn$ to $\lpn$.  For the reverse direction, let $\oZ:\lpn \to \lpn$ be an $\slnz$ contravariant and translation invariant Minkowski valuation.  By Theorem~\ref{contra}, we know that there is $c\ge 0$ such that $\oZ P=c\, \op P$ for every $P\in \lpn$. 
Since $\op T_{n-1} = 1/(n-1)!\,[-e_n, e_n]$, we conclude that $c\in(n-1)!\,\Z$.

\section{Equivariant valuations}

For $a,b\ge0$, define $\oZ_{a,b}: \lpn\to\ckn$ by
\begin{equation}\label{ozab}
\oZ_{a,b} P = a(P-\dst(P))+b(-P+\dst(P)).
\end{equation}
Note that $\oZ_{a,b}$ is an $\slnz$ equi\-variant and translation invariant Minkowski valuation on $\lpn$. In this section, we prove Theorem \ref{equi}, that is, we prove that for every $\slnz$ equi\-variant and translation invariant Minkowski valuation $\oZ$ on $\lpn$, 
there are $a,b\ge 0$ such that $\oZ=\oZ_{a,b}$.

\goodbreak
We deduce from Theorem~\ref{contra} in the planar case and from Lemma~\ref{planar-contra-equi} the planar case of Theorem~\ref{equi}.

\begin{prop}
\label{equi2}
If $\,\oZ: \lpt \to \ckt$ is an $\,\sltz$ equivariant and translation invariant  Minkowski valuation,  then there exist $a,b\geq 0$ such that $\oZ=\oZ_{a,b}$.
\end{prop}

As in the contravariant case, the following lemma is a simple consequence of the symmetry properties of $T_n$ and the equivariance of $\oZ$.

\begin{lemma}
\label{dimTco}
Let $\oZ: \lpn \to \ckn$ be an $\,\slnz$  equivariant and trans\-lation invariant Minkowski valuation.
If
$\,\oZ T_n \neq\{o\}$, then $o\in\Int (\oZ T_n)$.
\end{lemma}

The proof of the following result is analogous to the proof of  
Lemma~\ref{ZsimpleW} in the contravariant case.

\begin{lemma}
\label{ZsimpleWco}
Let $\oZ:\lpn\to \ckn$ be an $\,\slnz$ equivariant and trans\-lation invariant valuation. 
If $\,\oZ$ is simple, then  $\oZ [0,1]^{n}=\{o\}$.
\end{lemma}

\subsection{Lower dimensional polytopes}

In this section, we derive results on the image under a Minkowski valuation of lower dimensional lattice polytopes. The next lemma was proved in \cite{Haberl_sln, Ludwig:Minkowski} for $\sln$ equivariant (and homogeneous) valuations  on $\cpn$.
Let $n\ge 2$.

\begin{lemma}
\label{lowdim-equi}
If $\,\oZ:\lpn\to\ckn$ is an $\,\slnz$ equivariant  and translation invariant Minkowski valuation,  then
$\oZ P$ is contained in the subspace parallel to $\aff P$.
\end{lemma}
\begin{proof} By translation invariance and $\slnz$ equivariance, we may assume that $\Span P=\Span\{e_1,\dots,e_d\}$, where $d=\dim (P)\leq n-1$.

For $j\in\Z$ and $k\in\{d+1,\dots,n\}$, we define
$\phi_{jk}\in \slnz$ by $\phi_{jk}e_i=e_i$ for $i\neq k$ and  $\phi_{jk}e_k=e_k+je_1$. It follows that $\phi_{jk}P=P$. If $x=\sum_{i=1}^n t_i e_i\in \oZ P$, then
\begin{equation}
\label{skA}
\phi_{jk}x=(t_1+jt_k)e_1+\sum_{i=2}^{n}t_ie_i.
\end{equation}
Since $\phi_{jk}\oZ P=\oZ P$, the convex body $\oZ P$ is bounded, and $k\in\{d+1,\dots,n\}$ and $j\in\Z$ are arbitrary in (\ref{skA}), we conclude that $t_{d+1}=\dots=t_{n}=0$.
\end{proof}

\subsection{The cube }
\label{seccube-equi}

Proposition~\ref{cube-equi} is the main result of this section. 
Let $n\geq 2$.

\begin{prop}
\label{cube-equi}
If $\,\oZ: \lpn \to \ckn$ is an $\,\slnz$ equivariant and translation invariant  Minkowski valuation,  then
there exists $c\geq 0$ such that 
\begin{equation*}\label{mcube}
\oZ [0,1]^m=[- c, c]^m
\end{equation*}
for $m=0,\dots,n$.
\end{prop}

The critical step to prove Proposition~\ref{cube-equi} is the following
statement where $[a,b]^0=\{o\}$ for $a\le b$.

\begin{lemma}
\label{cube-equi-expand}
If $\,\oZ: \lpn \to \ckn$ is an \,$\slnz$ equivariant and translation invariant  Min\-kowski valuation  and 
there exists $c\ge 0$ such that
$\oZ [0,1]^m=[- c, c]^m$  for every  $m\leq n-1$, then
$$
\oZ(k[0,1]^{n})+k^n[- c, c]^n=k^n\oZ [0,1]^{n} +k\,[- c, c]^n
$$
for $k\in\N$.
\end{lemma}
\begin{proof}
 For $k\geq 1$, we subdivide $k[0,1]^{n}$ into a cell decomposition with $k^n$ unit cubes. We observe that for $m=1,\dots,n-1$,
$$
\sum_{1\leq i_1<\dots<i_m\leq n}\sum_{j=1}^m[-e_{i_j},e_{i_j}]=\binom{n-1}{m-1}[-1,1]^n,
$$
and for $m=0,\dots,n-1$, there exist $k^m(k-1)^{n-m}$ translates of $[0,1]^m$ that are faces of the cell decomposition intersecting the interior of $k[0,1]^{n}$. Since $\oZ [0,1]^m =[-c,c]^m$ for  $m\leq n-1$, we have $\oZ \{o\}=\{o\}$, and we deduce from Corollary~\ref{incl-ex}  that for $v\in \R^n$
\begin{align*}
h(&\oZ(k[0,1]^{n}),v)=\\
&=k^nh (\oZ [0,1]^{n}, v)+\sum_{m=1}^{n-1}(-1)^{n-m}\binom{n-1}{m-1}k^m(k-1)^{n-m} \,h([-c,c]^n,v)\\
&=k^nh(\oZ [0,1]^{n},v)+k\,\sum_{j=0}^{n-2}\binom{n-1}{j}k^j(1-k)^{n-1-j} \, h([-c,c]^n,v)\\
&=k^nh(\oZ [0,1]^{n},v)+k((k+1-k)^{n-1}-k^{n-1})\, h([-c,c]^n,v)\\[6pt]
&=k^nh(\oZ [0,1]^{n},v)-k^nh([{-c},c]^n,v)+k\,h([-c,c]^n,v).
\end{align*}
Thus the lemma is proved.
\end{proof}

\noindent{\bf Proof of Proposition~\ref{cube-equi}. }
We prove the statement by induction on $n\geq 2$. 
The case $n=2$ follows from  Proposition~\ref{equi2} and the fact that $\dst([0,1]^m)$ is the centroid of $[0,1]^m$ by Proposition~\ref{disc-steiner}.

Let $n\geq 3$ and  assume that Proposition~\ref{cube-equi} holds  for  $m\leq n-1$. We consider the $\slnz$ equivariant and translation invariant  Minkowski valuation $\oZ_n$ defined in Lemma~\ref{Minkowski-poly}. Since $\oZ_n$ is homogeneous of degree $n$, we deduce from
Lemma~\ref{cube-equi-expand} applied to $\oZ_n$ that $\oZ_n[0,1]^m=\{o\}$ for $m\le n-1$. Hence $\oZ_n$ is simple and we obtain by Lemma \ref{ZsimpleWco} that 
$$\oZ_n [0,1]^n= \lim_{k\to\infty}\frac{\oZ (k[0,1]^{n})}{k^n}=\{o\}.$$
Next, we apply Lemma~\ref{cube-equi-expand} to $\oZ$. Dividing both sides by $k^n$, and letting $k\to\infty$ shows that $\oZ [0,1]^{n}=[-c,c]^n$. \hfill $\qed$

\subsection{The prism}
\label{secprism}

Let $\oZ:\lpn \to \ckn$ be an $\slnz$ equivariant and translation invariant  Minkowski valuation. Let $n\geq 3$.

\begin{lemma}\label{prismSteiner}
If $S\in\lpm$ is a basic simplex and $k\in\N$, then 
\begin{equation*}
\cen(\oZ(S+[o,ke_n]))=o.
\end{equation*}
\end{lemma}
\begin{proof}
We may assume that $S=T_m$ for some $m=0,\dots,n-1$. 
If $m=0,1$, then $\phi(T_m+[0,ke_n])$ is a translate of $T_m+[0,ke_n]$ where $\phi\in\slnz$ is defined by
$\phi e_1=-e_1$, $\phi e_n=-e_n$ and $\phi e_j=e_j$ for $j=2,\dots,n-1$. Since we have
$\oZ(T_m+[0,ke_n])\subset\Span\{e_1,e_n\}$ by Lemma~\ref{lowdim-equi}, we deduce the statement of the lemma for $m=0,1$.

If $m\geq 2$, then $\oZ(T_m+[0,ke_n])\subset\Span\{e_1,\dots,e_m,e_n\}$ by Lemma~\ref{lowdim-equi}. Hence we may assume that $m=n-1$. Let $\alpha'\in\gln$ be the transformation that leaves $e_n$ fixed and acts on $\R^{n-1}$ as $\alpha$ defined in Section~\ref{secn-2}.
Then $ \alpha'\oZ(T_{n-1}+[o,ke_n])$ is invariant under the maps $\rho'$ that leave $e_n$ fixed and are orientation preserving isometries of the regular simplex $[v_0,\dots, v_{n-1}]\subset \R^{n-1}$ defined in Section~\ref{secn-2}. Thus the  first $(n-1)$ coordinates of the centroid of $\alpha'\oZ(T_{n-1}+[o,ke_n])$ vanish.
In addition, $\oZ(T_{n-1}+[o,ke_n])$ is invariant under the map $\psi\in\slnz$ defined 
by $\psi e_n=-e_n$, $\psi e_1=e_2$, $\psi e_2=e_1$, and $\psi e_j=e_j$ for $2<j<n$. This completes the proof of the lemma.
\end{proof}

\goodbreak

Recall that $\cyl=T_{n-1}+[0,e_n]$.

\begin{lemma}
\label{equitildeTn-1}
Assume that Theorem \ref{equi} holds true in dimension $(n-1)$ and hence  that
there exist $a,b\geq 0$ such that $\oZ P =\oZ_{a,b} P$ 
for every lower dimensional $P\in \lpn$. Then
$$\oZ \cyl=\oZ_{a,b} \cyl.$$
\end{lemma}
\begin{proof} We define  the convex body $\oZp P\subset\R^{n-1}$ for  $P\in\lpm$ by
\begin{equation*}
\oZp P=\pi_{e_n}\oZ(P+[o,e_n]).
\end{equation*}
Then $\oZp: \lpm \to \ckm$  is an $\slmz$ equivariant and translation invariant  Minkowski valuation.
Since Theorem \ref{equi} holds in dimension $(n-1)$,  there exist $a^\prime,b^\prime\geq 0$ such that 
\begin{equation}
\label{Ztilden-1}
\oZp P =\oZ_{a^\prime,b^\prime} P \mbox{ \ for $P\in \lpm$.}
\end{equation}
By Proposition~\ref{cube-equi}, we have
$$[-c,c]^{n-1}=\oZ [0,1]^{n-1} =\oZp [0,1]^{n-1} =\oZ_{a^\prime,b^\prime}[0,1]^{n-1}.$$
Combined with the assumption that $\oZ P =\oZ_{a,b} P$ 
for every lower dimensional $P\in \lpn$  and Proposition~\ref{disc-steiner} this gives
\begin{equation}
\label{alphatildealpha}
a^\prime+b^\prime=2c=a+b.
\end{equation}

For $m=(n-1)!$, we consider the triangulation $S'_1,\dots,S'_m$ of $[0,1]^{n-1}$
into $(n-1)$-dimensional basic simplices with $S'_1=T_{n-1}$ provided by Lemma~\ref{cube-triang}. For
$i=1,\dots,m$, set  $\widetilde{S}'_i=S'_i+[o,e_n]$. Note that the prisms  $\widetilde{S}'_1,\dots,\widetilde{S}'_m$ form a cell decomposition of $[0,1]^{n}$.
Let ${\cal F}$ denote the family of faces of the cell decomposition intersecting the interior of  $[0,1]^{n}$.
It follows from  the inclusion-exclusion principle that
\begin{equation}
\label{incl-exWequi}
\oZ [0,1]^{n}+\sum_{\genfrac{}{}{0pt}{}{F\in{\cal F}}
 {n-\dim (F)\;\text{odd}}}
\oZ F=
\sum_{\genfrac{}{}{0pt}{}{F\in\cal F}
 {n-\dim (F)\;\text{even}}}
\oZ F.
\end{equation}
We relate $\oZ$ to  $\oZ_{a,b}$. Note that
$\oZ_{a,b}$ in place of $\oZ$ also satisfies (\ref{incl-exWequi}). Since $\oZ [0,1]^{n}=\oZ_{a,b} [0,1]^{n}$ by Proposition~\ref{cube-equi}
and (\ref{alphatildealpha}), and $\oZ F=\oZ_{a,b} F $ for lower dimensional lattice polytopes $F\in\lpn$, we deduce that
\begin{equation}
\label{ZZprimeCi}
\sum_{i=1}^m\oZ \widetilde{S}'_i=\sum_{i=1}^m\oZ_{a,b} \widetilde{S}'_i.
\end{equation}
For  $i=1,\dots,m$, we have 
$$
\oZ_{a,b} \widetilde{S}'_i=a(S'_i-\cen(S'_i))+b(-S'_i+\cen(S'_i)) +c\,[-e_n,e_n] 
$$ 
by the Minkowski linearity of $\dst$, Proposition  \ref{disc-steiner} and (\ref{alphatildealpha}). Hence
the right hand side of (\ref{ZZprimeCi}) is of the form $Q+c\,m[-e_n,e_n]$ for a suitable $(n-1)$-dimensional polytope $Q\subset\R^{n-1}$.
It follows from (\ref{ZZprimeCi}) that $\oZ \cyl=\oZ \widetilde{S}_1'$ is an $n$-dimensional polytope that is a summand of  $Q+c\,m[-e_n,e_n]$. Hence, because of (\ref{sum-face}),  the facet outer normals of $\oZ \cyl$ are either parallel or ortho\-gonal to $e_n$.  Hence, by Lemma~\ref{prismSteiner}, there exists $c_0>0$ such that
$$
\oZ \cyl=\oZp T_{n-1}+c_0[- e_n, e_n].
$$
By (\ref{Ztilden-1}) and Proposition \ref{disc-steiner}, we therefore get
$$
\oZ \cyl=a^\prime(T_{n-1}-\cen(T_{n-1}))+b^\prime(-T_{n-1}+\cen(T_{n-1}))
+c_0[- e_n, e_n].
$$
Using the $\slnz$ equivariance  and translation invariance of $\oZ$, we deduce that
$$
\oZ \tilde{S}'_i=a^\prime(S'_i-\cen(S'_i))+b^\prime(-S'_i+\cen(S'_i))
+c_0[-e_n, e_n]
$$
for $i=1,\dots,m$. As $\sum_{i=1}^m\oZ \widetilde{S}'_i=Q+cm[-e_n,e_n]$, we conclude that $c_0=c$. Thus
(\ref{alphatildealpha}) implies
\begin{equation}
\label{ZtildeTalphatilde}
\oZ \cyl=a^\prime(\cyl-\cen(\cyl))+
b^\prime(-\cyl+\cen(\cyl)).
\end{equation}

To prove that $a=a^\prime$ and $b=b^\prime$, we first assume that $b\geq b^\prime$.
Define the vector $v=e_1+\dots+e_{n-1}$
and set $h_0=h({T_{n-1}-\cen(T_{n-1})},v)>0$. Note that
$$h({-T_{n-1}+\cen(T_{n-1})},v)=(n-1)h_0$$
and
$$
h({-\cyl+\cen(\cyl)},v)=(n-1)h_0
\mbox{ \ and \ }h({\cyl-\cen(\cyl)},v)=h_0.
$$
We consider the translation invariant real valued valuation $P \mapsto h({\oZ P},v)$ on $\lpn$ and for
$k\geq 2$, the cell decomposition of $T_{n-1}+[o,ke_n]$ into $k$ translates of $\cyl$.
The cell decomposition has $(k-1)$ faces intersecting the interior of $T_{n-1}+[o,ke_n]$, each a translate of $T_{n-1}$.
Since $\cen(\oZ(T_{n-1}+[o,ke_n]))=o$ by Lemma \ref{prismSteiner},
we have $h(\oZ(T_{n-1}+[o,ke_n]),v)\ge 0$. By first
applying the inclusion-exclusion principle (Corollary~\ref{incl-ex}), second that $\oZ T_{n-1}= \oZ_{a,b} T_{n-1}$
and (\ref{ZtildeTalphatilde}), and third that $a^\prime-a=b-b^\prime$, which follows from (\ref{alphatildealpha}), we deduce that
\begin{eqnarray*}
0&\leq &h({\oZ(T_{n-1}+[o,ke_n])},v)\\
&=& k\, h(\oZ \cyl,v)-(k-1)h({\oZ T_{n-1}},v)\\
&= &\big(k a^\prime+k (n-1)b^\prime-(k-1) a -(k-1)(n-1)b \big)h_0\\
&= &\big(a^\prime+(n-1)b^\prime-(k-1)(n-2)(b-b^\prime) \big)h_0.
\end{eqnarray*}
As $b\geq b^\prime$ and the last expression is non-negative for any large $k$, we conclude that $b=b^\prime$. In turn, $a=a^\prime$ follows from (\ref{alphatildealpha}).

If $b\leq b^\prime$,  and hence $a\geq a^\prime$, we use essentially the same argument,
only the valuation $P \mapsto h({\oZ P},-v)$ replaces $P\mapsto h({\oZ P},v)$, and we exchange the role of $a$ and $b$.
\end{proof}

\goodbreak
We deduce from  Lemma~\ref{equitildeTn-1} the following result.

\begin{coro}
\label{equitildeTn-1cor}
Let $\,\oZ$ and the constants $a,b$ be as in  Lemma~\ref{equitildeTn-1}. 
If $S_1,\dots,S_n$ with $S_1=T_n$ are  basic simplices
triangulating $T_{n-1}+[0,e_n]$,
then 
$$
\sum_{i=1}^{n} \oZ_{a,b} S_i=\sum_{i=1}^{n} \oZ S_i.
$$
\end{coro}

\subsection{The faces of $\mathbf{\oZ T_n}$}
\label{secnbigfacets}

Let $n\geq 3$. Let $\oZ:\lpn \to \ckn$ be an $\slnz$ equivariant and translation invariant  Minkowski valuation and 
assume that Theorem~\ref{equi} holds in $\R^{n-1}$. Hence, there exist $a,b\geq 0$ such that 
\begin{equation}
\label{equiTn0}
\oZ P=\oZ_{a,b} P
\end{equation}
for every lower dimensional $P\in \lpn$.  Note that Proposition \ref{cube-equi} implies that $2c=a+b$. 

If $a+b=0$, then $\oZ$ is simple. Hence, Lemma \ref{ZsimpleWco} implies that  $\oZ P=\{o\}$ for $P\in\lpn$. Thus the proof of Theorem \ref{equi} is complete in this case.

\begin{lemma}
\label{equiTn}
If $a+b>0$, then 
$\,\oZ T_n$ is an $n$-dimensional polytope with the property that any of its facet normals is also a facet normal of $\,T_n-T_n$.
\end{lemma}

\begin{proof} 
We use the triangulation $S_1,\dots,S_{n!}$ of $[0,1]^{n}$ into basic simplices
given by Lemma~\ref{cube-triang} with $S_1=T_n$.
Write ${\cal F}'$ for the faces of the cell decomposition that intersect the interior of $[0,1]^{n}$ and have dimension at most $n-1$.
We deduce from the inclusion-exclusion principle that
\begin{equation}
\label{Wtriang-equi}
\oZ [0,1]^{n} +\sum_{\genfrac{}{}{0pt}{}{F\in{\cal F}'}{n- \dim (F)\;\text{odd}}}
\!\!\!\oZ F=\sum_{i=1}^{n!}\, \oZ S_i +
\sum_{\genfrac{}{}{0pt}{}{F\in{\cal F}'}{ n-\dim (F)\;\text{even}}}\!\!\!
\oZ F.
\end{equation}

Note that (\ref{Wtriang-equi}) also holds for $\oZ_{a,b}$ in place of $\oZ$. Here $\oZ_{a,b} F=\oZ F$ for $F\in{\cal F}'$ by (\ref{equiTn0}). Therefore
(\ref{Wtriang-equi}) combined with $a+b=2c$ gives 
\begin{equation}
\label{Wtriang-equi0}
\sum_{i=1}^{n!} \oZ_{a,b} S_i =\sum_{i=1}^{n!} \oZ S_i.
\end{equation}
We deduce from $a+b>0$ that the left hand side of (\ref{Wtriang-equi0}) is $n$-dimensional. 
Since we have $S_i=\phi_i\,T_n$ with $\phi_i\in \slnz$ and $\oZ S_i=\phi_i\oZ T_n$, it follows that $\oZ T_n\neq\{o\}$.
We deduce from Lemma~\ref{dimTco} that $\oZ T_n$ is $n$-dimensional.

\goodbreak
We look at the decomposition of the unit cube $[0,1]^{n}$ into $T_n$ and the remaining part, $R_n$. 
Note that $T_n\cap R_n=[e_1,\dots,e_n]$. Using Proposition~\ref{cube-equi},  (\ref{equiTn0}),  and the valuation property of $\oZ$, we get
\begin{equation}
\label{TQW1}
\oZ T_n \,\,\mbox{ is a summand of }\,\, Q=[-c,c]^n+a\,[e_1,\dots,e_n]-b\,[e_1,\dots,e_n].
\end{equation}
We deduce right away that $\oZ T_n$ is a polytope.

Moreover,  (\ref{sum-face}) and (\ref{TQW1}) imply that any facet normal of $\oZ T_n$ is a facet normal of $Q$. 
Since $T_n-T_n$ is $o$-symmetric,  (\ref{sum-face}) implies that it is now sufficient to show that the affine hull of any facet $F$ of $[-1,1]^n+[e_1,\dots,e_n]-[e_1,\dots,e_n]$ is parallel to a facet of $T_n-T_n$. Now $F=F_0+F_1-F_2$ where $F_0$ is a face of $[-1,1]^n$, and $F_1$, $F_2$ are faces of $[e_1, \ldots, e_n]$. In particular,  
$$d_0+d_1+d_2 \geq n-1$$
where $d_i=\dim(F_i)$. If $d_0=0$, then $\aff F$ is a translate of $\aff [e_1,\dots,e_n]$ and hence $\aff F$ is parallel to a facet of $T_n-T_n$. Therefore we may assume that $d_0>0$, and, without loss of generality, that $\aff F_0$ is a translate of   $\Span\{e_1, \ldots, e_{d_0}\}$. 
Let $V_i\subset\{e_1,\dots,e_n\}$ be the set of vertices of $F_i$ for $i=1,2$. Since 
$$d_0+\card(V_1)+\card(V_2)= d_0+d_1+d_2+2\geq n+1,$$ 
we have $\{e_1, \ldots, e_{d_0}\}\cap(V_1\cup V_2)\neq\emptyset$, where $\card$ stands for cardinality. Hence we may assume that $e_1\in V_1$ and 
$\{e_1, \ldots, e_{d_0}\}\cup V_1=\{e_1,\dots,e_m\}$, where $m\leq d_0+d_1$. It follows from $e_1\in V_1$ that $\aff F_0+\aff F_1$ is a translate of 
$$\Span\{e_i:i=1,\dots,d_0\}+\Span\{e_i-e_1:e_i\in V_1\}=\Span\{e_i:i=1,\dots,m\}.$$
Therefore $\aff F$ is a translate of  the affine hull of the facet $[o, e_1,\dots,e_m]-F_2$ of $T_n-T_n$.
\end{proof}

\subsection{More on $\oZ T_n$}

As in Section \ref{secn-2}, set $T=T_n$ and $\alpha \bar T= T_*$. Define $Q_*=\alpha \oZ T$.
It follows from Lemma~\ref{equiTn} that
$Q_*$ is a polytope and that any of its facet normals  is
 a positive multiple of $\sum_{i\in I}v_i$ for a proper subset $I$ of $\{0,\dots,n\}$. Since $\oZ$ is $\slnz$ equivariant and translation invariant, $Q_*$ is invariant under maps from $\sym$. In particular,  the orbits of the facet normals of $Q_*$ of this action are characterized by the cardinality of $I$.

\begin{lemma}
\label{simplex-face}
For $m=1,\dots,n-2$, there exist constants $a_m,b_m\geq 0$ and $s_m\ge 0$  such that 
$$F(Q_*, w_m)=a_m F(T_*,w_m) +b_m F(-T_*,w_m)+s_m w_m.$$
\end{lemma}

\begin{proof} 
By Proposition~\ref{cube-equi} and (\ref{equiTn0}), the decomposition (\ref{TQW}) implies that\begin{equation}
\label{TQWn3}
Q_*\,\mbox{ is a summand of }\,(a+b) \sum_{i=1}^n [o, v_i-v_0]+a\,[v_1,\dots,v_n]- b\,[v_1,\dots,v_n].
\end{equation}
Let $m \in \{1,\dots,n-2\}$.
Since 
$$F(\sum_{i=1}^n [o, v_i-v_0],w_m) = \sum_{i=1}^n F( [o, v_i-v_0],w_m) = \sum_{i=1}^m [o,v_i-v_0],$$
and
$$F( \sum_{i=1}^n [o, v_i-v_0],-w_m) =  \sum_{i=1}^n F( [o, v_i-v_0],-w_m) = \sum_{i=m+1}^n [o,v_i-v_0],$$
we deduce 
from  (\ref{facets2}) and (\ref{TQWn3}) that
$
F(Q_*,w_m)$ is a summand of 
$$(a+b)\sum_{i=1}^m [o,v_i-v_0]+ a[v_1,\dots,v_m]-b\, [v_{m+1},\dots,v_n]$$
and
that $F(Q_*,-w_m)$ is a summand of 
$$(a+b)\sum_{i=m+1}^n [o,v_i-v_0]+ a[v_{m+1},\dots,v_n]-b\, [v_{1},\dots,v_m].$$
Lemma~\ref{summand} combined with (\ref{facets1}) implies that 
\begin{equation}
\label{F"betam2}
F(Q_*,w_m) = L_m-b_m F(T_*, -w_m)
\end{equation}
where $0\le b_m\le b$ and $L_m$  is a convex polytope  contained in a translate of  $\aff(v_0, \dots, v_m)$ and that
\begin{equation}
\label{F"betam3}
F(Q_*,-w_k) = a'_k F(T_*, -w_k) + L'_k
\end{equation}
where $0\le a'_k\le a$ and $L'_k$  is a convex polytope  contained in a translate of  $\aff(v_0, \dots, v_k)$.
For $\rho \in \sym$ suitable and $k=n-m-1$, we have $\rho Q_* = Q_*$ and $\rho (-w_k) = w_m$. Hence
(\ref{F"betam3}) implies
\begin{equation}
\label{F"betam4}
F(Q_*, w_m) = a'_{n-m-1} F(T_*, w_m) + \rho L'_{n-m-1}\end{equation}
where $\rho L'_{n-m-1}$  is contained in a translate of $\aff(v_{m+1}, \dots, v_n)$. Combining (\ref{F"betam2}) and (\ref{F"betam4}) with Lemma~\ref{summand} shows that
$F(Q_*, w_m)$ is a translate of  $a_m F(T_*,w_m) +b_m F(-T_*,w_m)$ with $a_m = a_{n-m-1}'$.

Thus there  are $c_0, \dots, c_n\in\R$  with $\sum_{i=0}^n c_i=1$ such that
\begin{equation}\label{fqe}
F(Q_*,w_m)= a_m [v_0, \dots, v_m]-b_m [v_{m+1}, \dots, v_n]+\sum_{i=0}^n c_i\,v_i.
\end{equation}
If $\rho\in\sym$ corresponds to an even permutation of $(v_0, \dots, v_m)$ and $(v_{m+1}, \dots, v_n)$, then
$\rho Q_*=Q_*$ and $\rho w_m=w_m$. Hence (\ref{fqe}) implies that
$$\sum_{i=0}^n c_i \rho \,v_i = \sum_{i=0}^n c_i\,v_i.$$
This implies that $c_0=\dots=c_m$ and $c_{m+1}=\dots=c_n$.
Thus $\sum_{i=0}^n c_i\,v_i=
(c_0-c_n)\, w_m$ and
 $$F(Q_*,w_m)= a_m F(T_*,w_m)-b_m F(T_*,-w_m)+s_m w_m$$
with $s_m=c_0-c_n$.
If $a_m, b_m>0$, then
 $a_m F(T_*,w_m)-b_m F(T_*,-w_m) $ is a facet of $a_m T_* - b_m T_*$. Thus (\ref{Xist}) implies that $s_m\geq 0$. If $a_m=0$ or $b_m=0$, then also $s_m\ge 0$.
\end{proof}

For $m=0,\dots,n-1$, we set $G_m= F\left(Q_*,w_m\right)$. Then,
for $m=1,\dots,n-2$, we have
\begin{equation*}
G_m=a_m [v_0,\dots,v_m]-b_m [v_{m+1},\dots,v_n]+s_m w_m.
\end{equation*}
If $a_m,b_m>0$, then $F(Q_*,w_m)$ is a facet of $Q_*$. Using (\ref{vpq0}), we obtain for its $(n-2)$-faces, 
\begin{eqnarray*}
F(G_m,w_{m-1})
&=&a_m [v_0,\dots,v_{m-1}]-b_m [v_{m+1},\dots,v_n]+s_m w_m,\\
F(G_m,w_{m+1})&=&
a_m [v_0,\dots,v_m]-b_m [v_{m+2},\dots,v_n]+s_m w_m.
\end{eqnarray*}
We need the following result.

\begin{lemma}
\label{mm+1m-1}
If $a_m,b_m>0$ for $m\in\{1,\dots,n-2\}$, then $\,G_{m-1}$ and
$\,G_{m+1}$ are facets of \,$Q_*$ and
\begin{eqnarray*}
F(G_m,w_{m-1})&=&F(Q_*,w_m)\cap F(Q_*,w_{m-1}),\\
F(G_m,w_{m+1})&=&F(Q_*,w_m)\cap F(Q_*,w_{m+1}).
\end{eqnarray*}
\end{lemma}

\begin{proof}
We only consider the case of $F(G_m,w_{m+1})$.
Since $F(G_m,w_{m+1})$ is an $(n-2)$-face of $Q_*$, we have
$$
F(G_m,w_{m+1})=F(Q_*,w_m)\cap F(Q_*,v)
$$
where $v=\sum_{i\in I}v_i$ for a proper subset
$I\subset\{0,\dots,n\}$.

On the other hand, $v$ is orthogonal to the affine hull of $F(G_m,w_{m+1})$. Therefore 
$$
v=s  w_m+t w_{m+1}=s(v_0+\dots+v_m)-t(v_{m+2}+\dots+v_n)
$$ 
for $s,t\in\R$. We deduce that $v\in\{\pm w_m,\pm w_{m+1},\pm v_{m+1}\}$. Readily $v\neq w_m$.
Since $Q_*$ is $n$-dimensional, and $F(G_m,w_{m+1})\subset F(Q_*,w_m)$, we have $v\neq -w_m$.
Next, $v\neq-v_{m+1}$ because
$a_m v_0-b_mv_{m+2}+s_m w_m\in F(G_m,w_{m+1})$ 
and $a_m v_0-b_mv_{m+1}+s_m w_m\in G_m$ 
but (\ref{vpq0}) implies that
\begin{eqnarray*}
h(Q_*,-v_{m+1})&\geq& -v_{m+1}\cdot (a_mv_0-b_mv_{m+1}+s_mw_m)\\
&>&  -v_{m+1}\cdot (a_mv_0-b_mv_{m+2}+s_mw_m).
\end{eqnarray*}
Next, $v\neq v_{m+1}$ because for $w=w_m-v_m+v_{m+1}=\sum_{i\in I}v_i$ 
corresponding to $I=\{0,\dots,m-1,m+1\}$, we have $a_mv_{m+1}-b_mv_{m+2}+s_mw\in F(Q_*,w)$,
and $s_m\geq 0$  gives
\begin{eqnarray*}
h(Q_*, v_{m+1})&\geq& v_{m+1} \cdot (a_mv_{m+1}-b_mv_{m+2}+s_m(w_m-v_m+v_{m+1}))\\
&>&  v_{m+1}\cdot (a_mv_0-b_mv_{m+2}+s_mw_m).
\end{eqnarray*}
Finally, $G_m= F(Q_*,w_m)$ is $(n-1)$-dimensional and 
$F(G_m,w_{m+1})$ is $(n-2)$-dimensional. Thus
$$
h(Q_*, -w_{m+1})\geq h({G_m},-w_{m+1})>h({F(G_m,w_{m+1})},-w_{m+1}),
$$
and hence $v\neq-w_{m+1}$. 
Thus $v=w_{m+1}$. 
\end{proof}

\begin{prop}
\label{ZStwo-faces} For $n=3$,
there exist $a_0,b_0,c_0\geq 0$ with the following properties.
If $S$ is a basic three-dimensional simplex and  $u,v\in\R^3\backslash\{o\}$ are such that $E=F(S,v)$ and
$E'=F(S,-v)$ are edges and $F=F(S,u)$ is a facet, then
\begin{eqnarray*}
F(\oZ S,u) &\text{ is a translate of }&a_0\,F+c_0(-F),\\
F(\oZ S,-u) &\text{ is a translate of }&b_0(-F)+c_0\, F,\\
F(\oZ S,v) &\text{ is a translate of }&a_0\,E+b_0\,E'.
\end{eqnarray*}
In addition, $c_0=0$ if and only if $\,\oZ S=\oZ_{a_0, b_0} S$.
\end{prop}

\begin{proof} We may assume that $S=T_3$ and write $T=T_3$. If \,$\oZ T=\{o\}$, then we have $a_0=b_0=c_0=0$. Otherwise, Lemma~\ref{dimTco} implies that $\oZ T$ is three-dimensional.
Note that by Lemma~\ref{equiTn} the facet normals of $Q_*$ are a subset of $\{\pm v_i: i=0,1,2,3\} \cup \{v_i+v_j: i\ne j\}$.

We claim that  if $z$ is a vertex of $Q_*$, then 
\begin{equation}
\label{vertex}
z\in F(Q_*, v_i+v_j) \text{ for some } i\neq j. 
\end{equation}
To prove (\ref{vertex}), we first assume that $a_1,b_1>0$. Then $F(Q_*,w_1)$ is two-dimensional. 
If the vertex $z$ lies in $F(Q_*,v_i)\cap F(Q_*,v_j)$ for $i\neq j$, then
$$F(Q_*,v_i+v_j)\subset \aff F(Q_*,v_i)\cap \aff F(Q_*,v_j).$$
This is not possible since $a_1,b_1>0$. 
Similarly, we see that for $i\neq j$ the vertex
$z\not\in F(Q_*,-v_i)\cap F(Q_*,-v_j)$. Since $z$ is contained in at least three two-dimensional faces of $Q_*$, $z\in F(Q_*,v_i+v_j)$
for some $i\neq j$.

Therefore we assume that either $a_1=0$ or $b_1=0$, that is, we have
$\dim F(Q_*,w_1)\leq 1$. 
In this case, we deduce from Lemma~\ref{equiTn} that any  exterior normal to  a two-dimensional face of $Q_*$ is an exterior normal  to a two-dimensional face of either $T_*$ or $-T_*$.  Hence
$$
Q_*=s\,T_*\cap (-t\,T_*) \mbox{ \ for $s,t>0$}.
$$
If $z$ is a vertex of $Q_*$, then it is not the midpoint of a segment contained in $Q_*$. Thus $z$ is contained in an edge of either $s\,T_*$ or $-t\,T_*$. Thus $z$ is a vertex of $F(Q_*,v_i+v_j)$ for some $i\neq j$, concluding the proof of (\ref{vertex}).

Recall that $s_1\geq 0$ if $a_1+b_1>0$. In addition $s_1>0$ if $a_1=b_1=0$ as $o\in{\rm int}\,Q_*$. 
We deduce from Lemma~\ref{simplex-face} and (\ref{vertex}), that
\begin{equation}\label{3DZS}
Q_*=[a_1 v_i-b_1 v_j+s_1(v_i+v_k):\,\{i,j,k\}\subset\{0,1,2,3\}].
\end{equation}
Thus it follows by a short calculation from (\ref{vpq0}) and (\ref{3DZS}) that  
\begin{align*}
F(Q_*,-v_0)&=a_1[v_0,v_1,v_2]-s_1[v_0,v_1,v_2]+ (b_1+s_1)(v_0+v_1+v_2),\\
 F(Q_*,v_0)&=-b_1[v_1,v_2,v_3]+s_1[v_1,v_2,v_3]+(a_1+s_1)v_0,\\
 F(Q_*,w_1)&=a_1[v_0,v_1]-b_1[v_2,v_3]+s_1(v_0+v_1). 
\end{align*}
Therefore we may choose $a_0=a_1$, $b_0=b_1$ and $c_0=s_1$. Since $F(T_*, -v_0)$ is a two-dimensional face and $F(T_*, \pm w_1)$ are edges, this concludes the proof.
\end{proof}

\subsection{Proof of Theorem~\ref{equi} for $\mathbf {n=3}$}
\label{secn3}

By Proposition~\ref{equi2},
 there exist $a,b\geq 0$ such that 
\begin{equation*}
\oZ P=\oZ_{a,b} P
\end{equation*}
for lower dimensional $P\in \lpth$. 
Let $S_1,S_2,S_3$ with $S_1=T_3$ be the basic simplices
triangulating the prism $\cylt=T_2+[0,e_3]$ defined in (\ref{Ttildedissect}). 
Corollary~\ref{equitildeTn-1cor}
 yields
\begin{equation}
\label{Ttildedissect00}
\oZ_{a,b} S_1 +\oZ_{a,b} S_2+\oZ_{a,b} S_3= \oZ S_1+ \oZ S_2+\oZ S_3.
\end{equation} 
We observe that $F(S_1,-e_3)=[o,e_1,e_2]$, that $F(S_2,-e_3)=[e_1,e_2]$, that
$F(S_2,e_3)$ is a translate of $[o,e_1]$ and that $F(S_3,e_3)$ is a translate of $ [o,e_1,e_2]$.
Set $G=F\left(\sum_{i=1}^{3} \oZ_{a,b} S_i,-e_3\right)$. Since $[o,e_1]$ is a translate of $[o,-e_1]$, we deduce that
\begin{equation}\label{G'sum3}
G \text{ is a translate of }
a[o,e_1,e_2]+b[o,-e_1,-e_2]+a[e_1,e_2]+b[o,e_1].
\end{equation}
By (\ref{Ttildedissect00}), we  also have $G=F\left(\sum_{i=1}^{3} \oZ S_i,-e_3\right)$. Hence Proposition~\ref{ZStwo-faces}  implies that $G$ is a translate of 
\begin{equation}
\label{Gsum3}
(a_0+c_0)[o,e_1,e_2]+(b_0+c_0)[o,-e_1,-e_2]+a_0[e_1,e_2]+b_0[o,e_1].
\end{equation}
Hence $F(G,-e_1)$ is a translate of $a[o,e_2]$ by (\ref{G'sum3}) and $F(G,-e_1)$ is a translate of $(a_0+c_0)[o,e_2]$ by (\ref{Gsum3}). Thus  $a=a_0+c_0$. 
From (\ref{G'sum3}) and (\ref{Gsum3}) we also obtain that $F(G,e_1)$ is a translate of $b[o,-e_2]$ and $F(G,e_1)$ is a translate of $(b_0+c_0)[o,-e_2]$, respectively. 
Hence $b=b_0+c_0$. 
Finally, we obtain that $F(G,-e_2)$ is a translate of $(a+b)[o,e_1]$ on the one hand, and is a translate of $(a_0+b_0+c_0)[o,e_1]$ on the other hand. 
Hence
$$
a_0+b_0+c_0=a+b=a_0+b_0+2c_0.
$$  
Therefore $c_0=0$, $a=a_0$ and $b=b_0$. Thus
$\oZ T_3 =\oZ_{a,b} T_3$ follows from Proposition~\ref{ZStwo-faces}.

\subsection{$\mathbf{\oZ T_n}$  for $\mathbf{n\geq 4}$}
\label{secnbig}

Let $\oZ$ be an $\slnz$ equi\-variant and translation invariant  Minkowski valuation on $\lpn$. Let $n\geq 4$. 

\begin{prop}
\label{ZTn4}
If  Theorem~\ref{equi} holds in $\R^{n-1}$, then
there exist $a_0,b_0\geq 0$ such that 
$$\oZ S=\oZ_{a_0,b_0} S$$
for every basic $n$-simplex $S$.
\end{prop}

\begin{proof}
Since $\oZ$ is $\slnz$ equivariant and translation invariant, it suffices to show there are $a_0,b_0\geq 0$ such that 
$$
\oZ T=\oZ_{a_0,b_0} T
$$
where $T=T_n$. 
By Lemma~\ref{equiTn}, we may assume that 
$\oZ T$ is an $n$-dimensional polytope. As before we set  $T_*=\alpha \bar T$ and $Q_*= \alpha \oZ T$. Thus we have to show that
that there are $a_0,b_0\geq 0$ such that
\begin{equation}\label{tsn}
Q_*=a_0 T_* -b_0 T_*.
\end{equation}

First,  let $\dim(F(Q_* ,w_m))\leq n-2$ for $m=1,\dots,n-2$. \\
Since Lemma~\ref{equiTn} implies that $Q_*$ is a polytope whose facet  normals are facet normals of $T_*-T_*$ and since $Q_*$  is invariant under the action of
$\sym$, we deduce that there are $s,t>0$ such that
\begin{equation}
\label{stXin}
Q_*=s\,T_*\cap(-t\,T_*).
\end{equation}
We claim that either
\begin{equation}
\label{stXinclaim}
\mbox{$s\geq nt$ or $t\geq ns$,}
\end{equation}
or in other words, either $-t\, T_*\subset s\, T_*$ or $s\,T_*\subset -t\, T_*$.
Suppose that (\ref{stXinclaim}) does not hold, that is,
\begin{equation*}
\label{stXinclaimno}
\frac sn<t<n\,s.
\end{equation*}
First, assume that $s\geq t$. Then $F(-t\, T_*,v_0+v_1)=[- t\,v_2,\dots,-t\,v_n]$ satisfies
$$
\cen(F(-t\,T_*,v_0+v_1))=t \,\frac{v_0+v_1}{n-1}\in t\Int T_*\subset s\Int  T_*.
$$
Thus (\ref{stXin}) implies that
\begin{equation}
\label{stXicut}
F(Q_* ,v_0+v_1)= s\,T_*\cap F(-t \,T_*,v_0+v_1)\,\mbox{ has dimension $n-2$. }
\end{equation}
We have $a_1=0$ and $b_1>0$, and
\begin{equation}
\label{stXitranslate}
F(Q_* ,v_0+v_1)\,\mbox{  is a translate of  $b_1 [-v_2,\dots,-v_n].$}
\end{equation}
However, as $s<nt$, the vertices $-tv_2,\dots,-tv_n$ of $F(-t\,T_*,v_0+v_1)$ are cut off by the
facets $F(s\,T_*,-v_i)$ of $s\,T_*$ for $i=2,\dots,n$. Therefore  (\ref{stXicut}) implies that
$F(Q_* ,v_0+v_1)$ has some $(n-3)$-dimensional faces with exterior normals
$-v_i$  for $i=2,\dots,n$. This contradicts (\ref{stXitranslate}), and in turn proves (\ref{stXinclaim}) if $s\geq t$.
Finally, the  case $s\leq t$ of (\ref{stXinclaim}) can be proved using the same argument for 
$$
F(Q_* ,-v_0-v_1)=F(s\,T_*,-v_0-v_1)\cap (-t\,T_*).
$$
It follows from (\ref{stXinclaim}) that either $Q_*=s\,T_*$, or $Q_*=-t\,T_*$. Thus (\ref{tsn}) holds in this case.

Second, let $F(Q_* ,w_m)$ be $(n-1)$-dimensional for some $m=1,\dots,n-2$. \\
It follows from  Lemma~\ref{mm+1m-1} that $\dim (F(Q_* ,w_k))= n-1$ for $k=0,\dots,n-1$.
It also follows from  Lemma~\ref{mm+1m-1} that if $m=1,\dots,n-2$, then
\begin{align}\label{FTm-1n-2}
F(Q_*,w_{m-1})&\cap F(Q_* ,w_m)\\
\nonumber
&=a_m [v_0,\dots,v_{m-1}]-b_m[v_{m+1},\dots,v_n] +s_m w_m,\\[4pt]
\label{FTm+1n-2}
F(Q_*,w_m) &\cap F(Q_* ,w_{m+1})\\
\nonumber
&=a_m[v_0,\dots,v_{m}]-b_m[v_{m+2},\dots,v_n]+s_m w_m.
\end{align}
If $m\geq 2$, then (\ref{FTm+1n-2}) applied to $F(Q_*,w_{m-1})\cap F(Q_*,w_m)$ shows that
\begin{align}
\label{FTm-1n-2sm-1}
F(Q_*,& w_{m-1})\cap F(Q_*, w_m)\\
\nonumber
&=a_{m-1}[v_0,\dots,v_{m-1}]-b_{m-1}[v_{m+1},\dots,v_n] +s_{m-1}w_{m-1}.
\end{align}
Comparing (\ref{FTm-1n-2}) and (\ref{FTm-1n-2sm-1}) 
implies $a_{m-1}=a_m$, $b_{m-1}=b_m$ and $s_m=s_{m-1}=0$.

Similar arguments based on  (\ref{FTm+1n-2}) prove that 
if $m\leq n-3$, then 
$a_{m+1}=a_m$, $b_{m+1}=b_m$ and $s_{m+1}=0$. Continuing step by step, 
we conclude that $a_1= \dots=a_{n-1}$, $b_1=\dots =b_{n-1}$ and $s_1=\dots=s_{n-1}=0$.

Set
$a_0=a_1$ and $b_0=b_1$. Since
 $s_k=0$ for $k=1,\dots,n-1$,  we obtain from   Lemma \ref{simplex-face} that
$$
F(Q_*, w_k)= F(a_0 T_* -b_0 T_*, w_k)
$$
for $k=1,\dots,n-2$.
It  follows from  Lemma~\ref{mm+1m-1} that
$$h(Q_*,w_k)=h(a_0 T_*-b_0 T_*, w_k)$$
for $k=0$ and $k=n-1$ as well. By symmetry,  the support functions
of $Q_*$ and $a_0T_*-b_0T_*$ agree for any possible facet normal of either polytope. Thus we conclude that (\ref{tsn}) holds.
\end{proof}

\subsection{Proof of Theorem~\ref{equi} for $\mathbf{n\geq 4}$}
\label{secnatleast4}
Let $\oZ: \lpn \to\ckn$ be an $\slnz$ equi\-variant and translation invariant  Minkowski valuation.
We prove  Theorem~\ref{equi} by induction on the dimension $n\geq 3$. The case $n=3$ is settled in Section~\ref{secn3}. 
Therefore we assume that $n\geq 4$ and  that Theorem~\ref{equi} holds in dimension $(n-1)$. 
In particular,
there exist $a,b\geq 0$ such that
\begin{equation*}
\label{equiTn4}
\oZ P=\oZ_{a,b} P
\end{equation*}
for lower dimensional $P\in \lpn$ (where $\oZ_{a,b}$ is defined in (\ref{ozab})).  In addition, $a=b=0$ implies that $\oZ P=\{o\}$ for $P\in\lpn$,
 and $a+b>0$ implies that 
$\oZ T_n$ is an $n$-dimensional polytope. 

We may assume that $a+b>0$, and hence  Proposition~\ref{ZTn4} implies the existence of $a_0,b_0\geq 0$
with $a_0+b_0>0$ such that if $S$ is a basic $n$-simplex, then
\begin{equation}
\label{equiS4}
\oZ S =\oZ_{a_0,b_0} S.
\end{equation}
We compare $\oZ$ and $\oZ_{a,b}$.

Let $S_1,\dots,S_n$ with $S_1=T_n$ be the basic simplices
triangulating the prism $\cyl=T_{n-1}+[0,e_n]$ in (\ref{Ttildedissect}).  Corollary~\ref{equitildeTn-1cor}
implies that
\begin{equation}
\label{sumZZ'}
\sum_{i=1}^{n} \oZ_{a,b} S_i=\sum_{i=1}^{n} \oZ S_i.
\end{equation}
Let $1>r_3>\dots>r_{n-1}>0$ and 
$$
w=-e_n+e_1+e_2+\sum_{i=3}^{n-1}r_i e_i.
$$
It follows that $
F(S_i,-w)=\{e_n\}$ for $i=1,\dots,n$
and that 
$F(S_1,w)=[e_1,e_2]$, $F(S_2,w)=[e_1,e_2]$, and $F(S_i,w)=\{e_{i-1}\}$ for $i=3,\dots,n$.
We deduce from the definition of $\oZ_{a,b}$ and (\ref{equiS4})  that
\begin{align*}
F\big(\sum_{i=1}^{n} \oZ_{a,b} S_i,w\big)&\mbox{ is a translate of }2a\,[e_1,e_2],\\
F\big(\sum_{i=1}^{n} \oZ S_i,w\big)&\mbox{ is a translate of }2a_0\,[e_1,e_2],
\end{align*}
and hence $a=a_0$ follows from (\ref{sumZZ'}). Similarly
\begin{align*}
F\big(\sum_{i=1}^{n} \oZ_{a,b} S_i,-w\big)&\mbox{ is a translate of }2b\,[e_1,e_2],\\
F\big(\sum_{i=1}^{n} \oZ S_i,-w\big)&\mbox{ is a translate of }2b_0\,[e_1,e_2],
\end{align*}
and hence $b=b_0$. Therefore Corollary~\ref{BetkeKneserinvariance} implies that
$\oZ =\oZ_{a,b} $ on $\lpn$. 

\subsection{Proof of Theorem \ref{equii}}

Set $m_n=\lcm(2,\dots, n+1)$.
Proposition \ref{disc-steiner} implies that $m_n\, \dst(P)\in \Z^n$ for $P\in\lpn$. Hence, for integers $a,b\ge 0$ with $b-a\in m_n \,\Z$, the operator $\oZ$ defined by
$$P \mapsto a\,(P-\dst(P))+b\,({-P}+\dst(P))$$
maps $\lpn$ to $\lpn$. For the reverse direction, let $\oZ:\lpn \to \lpn$ be an $\slnz$ equivariant and translation invariant Minkowski valuation.  By Theorem \ref{contra}, we know that there are $a,b\ge 0$ such that
$$\oZ P= a\,(P-\dst(P))+b\,({-P}+\dst(P))$$
for every $P\in \lpn$.
Since $\oZ P\in \lpn$ for all $P\in\lpn$, setting $P=T_k$ and using Proposition \ref{disc-steiner} shows that 
$$a\big(T_k -\frac{e_1+\dots+e_k}{k+1}\big) + b\big(-T_k+\frac{e_1+\dots+e_k}{k+1}\big)\in\lpn.$$
Hence
$a+c/(k+1), -(a+c)+c/(k+1) \in \Z$ for $k=1,\dots, n$ with $c=b-a$. Thus $c=b-a \in m_n \,\Z$ and $a,b\in \Z$.

\section*{Acknowledgments} The authors thank Raman Sanyal for pointing out Corollary \ref{polynomiality-coefficient1} and its proof to them, Peter McMullen for an improvement of the proof of Lemma \ref{zonotope}  and Imre B\'ar\'any and P\'al Heged\H{u}s for valuable discussions. They also thank the referees for their helpful remarks.

The work of K\'aroly J.~B\"or\"oczky was supported, in part,  by the Hungarian Scientific Research Funds No 109789 and No 116451. The work of Monika Ludwig was supported, in part, by Austrian Science Fund (FWF) Project P25515-N25.

\bigskip
\parindent=0pt
\begin{samepage}
K\'aroly J.~B\"or\"oczky\\
Alfr\'ed R\'enyi Institute of Mathematics\\
Hungarian Academy of Sciences\\
1053 Budapest, Re\'altanoda u. 13-15\\
Hungary\\
E-mail:	 carlos@renyi.hu
\end{samepage}

\bigskip

\begin{samepage}
Monika Ludwig\\
Institut f\"ur Diskrete Mathematik und Geometrie\\
Technische Universit\"at Wien\\
Wiedner Hauptstra\ss e 8-10/1046\\
1040 Wien, Austria\\
E-mail: monika.ludwig@tuwien.ac.at
\end{samepage}

\end{document}